\documentclass[10pt,regno]{amsart}
\usepackage{latexsym}
\usepackage{mathtools}
\usepackage{overpic}
\usepackage{amsmath,amsthm,amsfonts,amssymb,amscd,mathrsfs}
\usepackage{euscript,graphicx,color}
\usepackage{epstopdf,rotating}
\usepackage{chngcntr}
\usepackage{apptools}
\pdfoutput=1
\newtheorem*{proposition*}{Proposition}
\newtheorem*{theorem*}{Theorem}
\newtheorem{theorem}{Theorem}

\newtheorem{defi}{Definition}[section]
\newtheorem{teo}[defi]{Theorem}

\newtheorem{cl}[defi]{Claim}

\newtheorem{remark}[defi]{Remark}

\newtheorem{lemma}[defi]{Lemma}

\numberwithin{equation}{section}

\DeclareMathOperator{\I}{I}
\DeclareMathOperator{\J}{J}

\def\L{\mathfrak{L}}
\def\BHa{{\bf{BH1}}}
\def\BHb{{\bf{BH2}}}
\def\BHc{{\bf{BH3}}}
\def\BHd{{\bf{BH4}}}
\def\BHe{{\bf{BH5}}}
\def\BHf{{\bf{BH6}}}
\newcommand{\eqdef}{\stackrel{\scriptscriptstyle\rm def}{=}}

\begin{document}

\title[Blender-horseshoes]{ Blender-horseshoes\\ in center-unstable H\'enon-like families}
\author[L. J. D\'iaz and S. A. P\'erez ]{Lorenzo J. D\'iaz and Sebasti\'an A. P\'erez}
\address{Departamento de Matem\'atica PUC-Rio, Marqu\^es de S\~ao Vicente 225, G\'avea, Rio de Janeiro 225453-900, Brazil}
\email{lodiaz@mat.puc-rio.br}
\address{Centro de Matem\'atica da Universidade do Porto, Rua do Campo Alegre, 687, 4169-007 Porto, Portugal}
\email{sebastian.opazo@fc.up.pt}

\begin{abstract}
A \textit{blender-horseshoe} is a locally maximal transitive hyperbolic set that  appears in dimension at least
three 
carrying a distinctive geometrical property:
its local stable manifold ``behaves'' as a manifold
of topological dimension greater than the expected  one (the dimension of the stable bundle). This property persists under  perturbations turning this kind of dynamics an important piece in the global description of robust non-hyperbolic systems.  
In this paper, we consider a parameterized family of center-unstable H\'enon-like of endomorphisms in dimension three and show how
blender-horseshoes naturally occur in a specific parameter range.
\end{abstract}
  
\thanks{This paper is part of the PhD thesis of SP (PUC-Rio) supported by CNPq (Brazil).
The authors  thank the hospitality and support of Centro de Matem\'atica of Univ. of Porto (Portugal).
LJD is partially supported by CNE-Faperj, CNPq-grants (Brazil) and SP is partially supported by CMUP (UID/MAT/00144/2013) and PTDC/MAT-CAL/3884/2014, which are funded by FCT (Portugal) with national (MEC) and European structural funds through the programs  COMPTE and FEDER, under the partnership agreement PT2020.
}

\keywords{Blender, Blender-horseshoe, H\'enon-like families}
\subjclass[2000]{
%37D25, %Nonuniformly hyperbolic systems (Lyapunov exponents, Pesin theory, etc.)
%37D35, % Thermodynamic formalism, variational principles, equilibrium states
%28D20, % Entropy and other invariants
37C45, % Dimension theory of dynamical systems
%28D99, % Measure-theoretic ergodic theory
%37F10 % Rational maps
37D30, % partially hyperbolic systems and dominated splittings
37C29%Homoclinic and heteroclinic orbits
}

\maketitle

\begin{center}
\hfill \emph{To Welington de Melo, in memoriam}\\
\end{center}

\section{Introduction}

\label{s.introduction} 
Naively, a {\emph{blender}} is a transitive hyperbolic set  that appears in dimension at least three and whose special geometrical configuration implies that 
the ``dimension" of its stable set is larger than the ``expected" one. To be a bit more precise, recall that the {\emph{index}} of a transitive hyperbolic set $\Lambda$, denoted by $\mathrm{ind}(\Lambda)$, is the dimension of its stable bundle
(by transitivity, the index is well defined). The leaves of the (local) stable sets of points in $\Lambda$ have dimension $\mathrm{ind}(\Lambda)$, however the (local) stable set of the blender $\Lambda$ behaves as a set of dimension $\mathrm{ind}(\Lambda)+1$ (or greater).
 In practical terms and applications, blenders are dynamical ``local plugs" which in some  (semi-local or global) configurations carry further important  properties of the dynamics (see the next paragraph). For an informal presentation of blenders 
and a discussion on their role
 in smooth dynamical systems we refer to ~\cite{BDCW} and \cite[Chapter 6.2]{BDV}.
Blenders were introduced in \cite{BonDia:96} as a
formalisation of the constructions in~\cite{L95}
in the context of  bifurcations via heterodimensional cycles. In \cite{BonDia:96},
blenders were used to construct new classes
of robustly transitive diffeomorphisms.  Later, blenders were used in several dynamical contexts:
Generation of robust heterodimensional cycles and  homoclinic tangencies, stable ergodicity,  Arnold diffusion, and construction of nonhyperbolic measures, among others. Each of these applications involves a specific
 type of blender such as blender-horseshoes \cite{BDtang}, symbolic blenders \cite{NasPuj:12,BarKiRai:14}, dynamical blenders \cite{BocBonDia:16}
and super-blenders~\cite{AviCroWil:16}.

In the original definition in  \cite{BonDia:96} the main emphasis is placed
on the persistence of its geometrical configuration that was key to guarantee the robust transitivity of non-hyperbolic sets, see the discussion in \cite[Chapter 6]{BDV}.  Although in many contexts the ``original" blenders in \cite{BonDia:96} are shown to be very useful,  a major con of  them  is that 
they fail to be  locally maximal sets, this deficiency carries some  constraints in their use and applications. 
This weakness was bypassed in
\cite{BDtang} by introducing a special type of blenders, called {\emph{blender-horseshoes,}} which are locally maximal and also conjugate to the standard Smale horseshoe, see Definition~\ref{d.BH}. These two additional useful properties can be explored to get additional relevant properties:
% the local maximality of these blenders has important consequences and 
blender-horsehoes are the key local plugs to get {\emph{robust heterodimensional cycles}} and {\emph{robust homoclinic tangencies}} in the $C^1$-topology, see \cite{BDcycles} and \cite{BDtang}. 
 In some cases, one can also get some extra ``fractal-like" information about these blenders, see \cite{DiaGelGroJag:17} and also 
 \cite{MorZul:2012}. 
 Considering these aspects and also the use of blenders  to get robust cycles in bifurcation theory, one can think of blender-horseshoes as  a version of the 
so-called {\emph{thick horseshoes}} introduced by Newhouse in the construction of robust homoclinic tangencies of surface diffeomorphisms, 
see~\cite{N}.

In what follows, for simplicity and also considering the scope of this paper, our discussion is restricted to the three-dimensional case (adjustments to higher dimensions are straightforward).
There are some settings where blender-horseshoes appear in a natural way. A first one is the bifurcation of {\emph{heterodimensional cycles}} (i.e., there are a pair of
saddles having  indices one and two whose invariant manifolds meet cyclically). In this context,
the occurrence of blender-horseshoes is related to the existence of some non-normally hyperbolic dynamics that can be illustrated as
follows. Think of a standard horseshoe defined on
a ``square" and ``multiply" this dynamics by a ``weak expansion"  in the normal direction (to the square), see Figure \ref{fig:N}.
%\begin{figure}[h] 
% \centering
%\begin{overpic}[scale=.12,bb=0 1 766 1447,tics=5
%]{normal.jpg}
  %\end{overpic}
 %\caption{Non-normally hyperbolic dynamics.}
 % \label{fig:N}
%\end{figure}
\begin{figure}[h]
\centering
 \includegraphics[width=.3
 \textwidth]{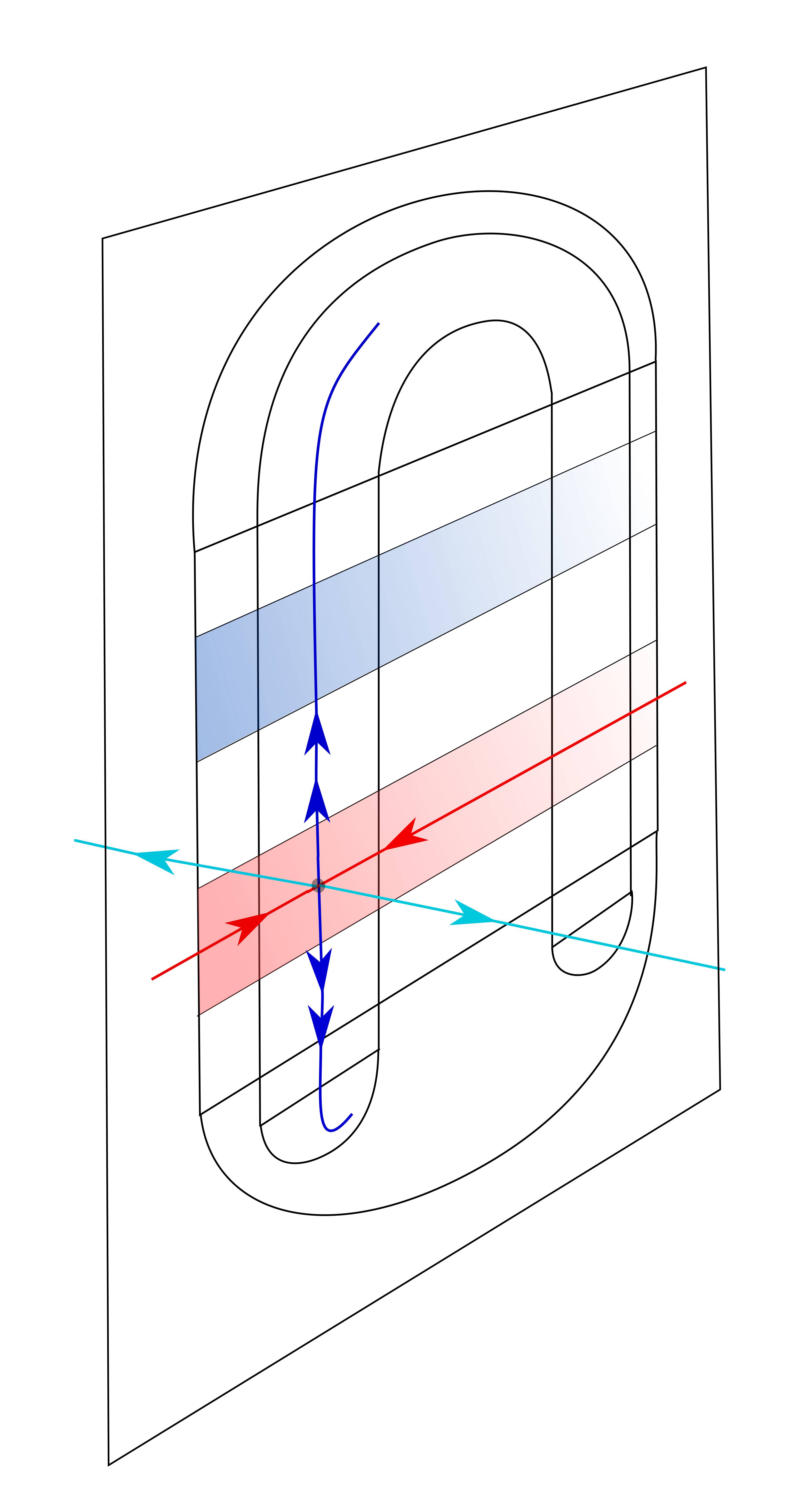}
 \caption{Non-normally hyperbolic dynamics.} \label{fig:N}
\end{figure}

In this way, one gets a hyperbolic set (of index one)  contained in a non-normally hyperbolic (local) manifold. Persistence of hyperbolicity implies that this horseshoe has continuations for small perturbations of the dynamics. However, since the horseshoe is contained in a non-normally hyperbolic square, the new horseshoes are in general not contained in a local surface. It turns out that
appropriate perturbations of the initial dynamics provide blender-horseshoes.
% Indeed, these blenders are the prototypical examples of blender-horseshoes.
For a complete discussion of this construction (and also with explicit formulae) we refer to~\cite{BonDiaVia:95}
(note that in \cite{BonDiaVia:95} the term blender is not used).

An interesting question is to provide explicit examples of maps (with an explicit analytic formula) 
exhibiting  blender-horseshoes. This leads to   
the second ingredient of this paper, a family of endomorphisms so-called 
{\emph{center-unstable H\'enon-like families,}} see equation~\eqref{e.henonlike}.
We recall that in the two-dimensional case, H\'enon-like maps are a fundamental ingredient in the study of homoclinic bifurcations
which provide
 a ``limit dynamics'': there exists a sequence of bifurcation parameters providing a sequence of return maps at the homoclinic tangency converging to  
a  H\'enon-like map in suitable 
rescaled coordinates. This construction, known as \textit{renormalisation scheme}, when performed at  homoclinic tangencies
allows to translate (robust) properties of the H\'enon-like family 
to the dynamics of diffeomorphisms nearby the bifurcating one, for details see~\cite[Chapter 3]{PalTak1995}. Two remarkable
examples of
such portable properties are the persistence of homoclinic tangencies \cite[Chapter 3]{PalTak1995} and the existence on strange attractors \cite{MV}. 

In view of  the above discussion, it is natural to ask about renormalisation schemes and limit dynamics in  heterodimensional settings. In this direction, in~\cite{DKS} it is considered a heterodimensional cycle
(associated to a pair of saddles of indices one and two)
involving a  heteroclinic orbit corresponding to the tangential contact of the two-dimensional invariant manifolds  of the saddles.
 This heteroclinic orbit is called a \textit{heterodimensional tangency}, see \cite{DNP}. In \cite{DKS}
 it is provided a renormalisation scheme 
 %at the heterodimensional tangency 
 whose limit dynamics is a center-unstable H\'enon-like family.
This discussion justifies the following % It is now opportune a 
technical remark. On the one hand, the theory of homoclinic bifurcations and renormalisation schemes requires at least $C^2$-regularity
 of the diffeomorphisms\footnote{Besides the regularity of the maps, necessary for the convergence of the renormalisation scheme, 
 another key fractal-like ingredient is the thickness of a hyperbolic set, which has a radically different behaviour in the $C^1$ and $C^2$-topologies,
 see~\cite{Ure:1995} and ~\cite{MO}.}.
 On the other hand, the construction of robustly  non-hyperbolic dynamics (robust cycles and tangencies)  associated to heterodimensional cycles is mostly developed in the $C^1$-case\footnote{The starting point of this progress is due to the development of a series of typically $C^1$-tools  (started with Pugh's $C^1$ closing lemma and with Franks derivative perturbation lemma) that to the current date have no equivalents in $C^r$-topologies with $r>1$. On the other hand, $C^1$-regularity is not sufficient to some results requiring control of the distortion.}.
%As classic example we cite the one-dimensional Denjoy's theory.}. 
Thus, an interesting problem is to develop these theories in higher regularity.

 First, for direct approach dealing with perturbation of product dynamics (a hyperbolic part times the identity) we refer to
 ~\cite{BR:2017}.  
On the other hand, bifurcations of heterodimensional tangencies 
seem to be an appropriate setting
for obtaining robustly non-hyperbolic dynamics in high regularity, see for instance
% These bifurcations were introduced in~\cite{DNP} providing natural transitions from partially hyperbolic dynamics to $C^1$-robust non-dominated dynamics. \textcolor{magenta}{Moreover, the construction in~\cite{DNP}  can lead to the $C^1$-Newhouse phenomenon and to $C^1$-universal dynamics, see~\cite{BD2003}.} \margem{too much information....}
% In higher regularity, 
%we mentions the examples in
\cite{KS:2012} where
 $C^2$-robust heterodimensional tangencies and $C^2$-robust heterodimensional cycles involving heterodimensional tangencies are obtained using blenders and the results of~\cite{PV}. 
Our results are motivated by the ideas of \cite{DKS}, where
blenders are generated at the bifurcation of heterodimensional cycles in high regularity topologies.
More precisely, in   \cite{DKS} 
 blender   are obtained for some (open) range of parameters of the
center-unstable H\'enon-like family  and some applications (involving a renormalisation scheme) are given for the bifurcation of heterodimensional cycles  
 in high regularity (in the spirit of
\cite{PalTak1995}). 
In this paper, we prove that the blenders
obtained in~\cite{DKS} are indeed blender-horseshoes. This step will allow (in further applications) to improve versions of
\cite[Theorem 1.4]{DKS}, getting robust cycles and robust tangencies in higher regularity (in the same spirit as in \cite{BDcycles,BDtang}). 
In a forthcoming  paper (see also \cite{Perezthesis}) we will 
introduce a renormalisation scheme 
for some non-transverse heterodimensional cycles (cycles with heterodimensional tangencies)
converging to the center-unstable H\'enon-like family \ref{e.henonlike}and state the persistence of cycles and tangencies (in higher regularity) after its bifurcation.

Finally, let us observe that  \cite{HKOS:2018}  
provides a quite complete numerical analysis of the center-unstable H\'enon family in \eqref{e.henonlike},
showing strong numerical evidences of the occurrence of  blenders in a parameter range wider than the one in \cite{DKS} and  illustrates the vanishing of these blenders beyond this range.
We believe that
the blenders detected in \cite{HKOS:2018} are indeed blender-horseshoes.

It follows the main result of this paper.

\begin{theorem}
\label{t.BH-DKS}
Consider the center-unstable H\'enon-like family of endomorphisms
\begin{equation}
\label{e.henonlike}
G_{(\xi,\mu,\kappa,\eta)}(x,y,z) \eqdef (y,\mu+y^2+\kappa\,y\,z+\eta\,z^2,\xi\,z+y), \quad \xi>1.
\end{equation}
Then there is  $\varepsilon>0$ such that
for every 
$$
\bar \nu =(\xi, \mu, \kappa,\eta)\in 
\mathcal{O}_\varepsilon\eqdef (1.18,1,19)\times (-10,-9)\times (-\varepsilon,\varepsilon)^2
$$
the 
endomorphism $G_{\bar\nu}$
has a blender-horseshoe in the cube 
 $ \Delta\eqdef [-4,4]^2\times[-40,22].$

As a consequence, 
every diffeomorphism or endomorphism sufficiently $C^1$-close to $G_{\bar \nu}$ has a 
blender-horseshoe in $\Delta$.
\end{theorem}

The consequence pointed out in the theorem arises from the $C^1$-persistence of blenders, see Remarks~\ref{r.golpedeestado}
and \ref{r.electionday}. Let us observe that this result is a version of \cite[Theorem 1.1]{DKS} where blenders are replaced by blender-horseshoes in a similar range of parameters.
 
This paper is organised as follows. In Section~\ref{ss.blenders}, we introduce the definitions of blender and blender-horseshoe and state the distinctive property of a blender-horseshoe (Lemmas~\ref{l.sp} and \ref{l.dyn-geo}). In Section \ref{s.localperturbationsgk},
we prove Theorem~\ref{t.BH-DKS}.

\section{Blenders and Blender-horseshoes}
\label{ss.blenders}
\subsection{Blenders}
The notion of a cu-blender (or simply blender) was introduced in~\cite{BonDia:96}, where were used
to generate $C^1$-robust transitivity in the non-hyperbolic setting. 
The main virtue of a blender comes from its special internal geometry: a cu-blender is a transitive hyperbolic set whose (local) stable set robustly “behaves” as manifold of topological dimension larger than the dimension of its stable bundle. We now discuss the (axiomatic) definition of blenders in the three-dimensional case.

\begin{defi}\label{d.blender}
{
\em{($\mathrm{cu}$-Blender, Definition 3.1 in \cite{BDtang})
Let $f:M\to M$ be a three-dimensional diffeomorphism. A transitive hyperbolic compact set $\Lambda$ of index two of $f$  is a cu-\textit{blender} if
there are a 
$C^1$-neighbourhood $\mathcal{U}$ of $f$ and a 
$C^1$-open set $\mathcal{D}$ of embeddings
of one-dimensional discs $D$ into $M$ such that for every $g \in \mathcal{U}$ and every disc $D \in \mathcal{D}$ 
the local stable manifold $W^{\mathrm s}_{\mathrm{loc}}(\Lambda_g)$ of the continuation $\Lambda_g$ intersects $D$. The set $\mathcal{D}$ is called the \textit{region of superposition} of the blender.
}}
\end{defi}
\subsection{Blender-horseshoes}
This kind of blenders was introduced in~\cite{BDtang} as  a mechanism for the generation of $C^1$-robust tangencies in dimension equal to or greater  than three. 
Comparing with the standard blenders, blender-horseshoes satisfy the following additional property:  they are locally maximal invariant sets  conjugate to a complete shift of two symbols. 
These properties 
provide a complete description of its local stable manifold  as well as a nice  geometrical structure:
the local stable manifold of a blender-horseshoe is the Cartesian product of a ``fat Cantor set" by an ``interval", see Remark~\ref{r.golpedeestado}.
We now give the definition of a blender-horseshoe following \cite[Section~3.2]{BDtang}, for further details we refer to that paper. 
As the construction is local, we assume that the ambient space is
$\mathbb{R}^3$. 
We start with some preliminary definitions.

For $a>0$ consider the
interval $\mathrm I_a\eqdef [-a,+a]$ and for $x,y,z\in \mathbb{R}^+$ the
 cube 
 $$
\Delta \eqdef \mathrm{I}_x \times \mathrm I_y \times \mathrm I_z \subset \mathbb{R}^3.
%\quad
%\mbox{where $\mathrm I_x,\mathrm I_y$, and $\mathrm I_z$ are closed intervals.} 
$$
We divide the  boundary $\partial\Delta$  of $\Delta$ into three parts 
as follows:
$$
\partial^{\mathrm{s}}\Delta \eqdef \partial \mathrm I_x \times\mathrm I_y \times\mathrm I_z,
\quad \partial^{\mathrm{uu}}\Delta 
\eqdef  \mathrm I_x \times \partial\mathrm I_y \times\mathrm I_z,
\quad \partial^{\mathrm{u}}\Delta 
\eqdef \mathrm I_x \times \partial(\mathrm I_y \times \mathrm I_z).
$$
Note that $\partial \Delta= \partial^{\mathrm{s}}\Delta\cup \partial^{\mathrm{u}}\Delta$ and
$\partial^{\mathrm{uu}}\Delta
\subset \partial^{\mathrm{u}}\Delta$.

Given
$\theta>0$ and $p\in\mathbb{R}^3$, 
define the $\mathrm{s}$-, $\mathrm{uu}$- and $\mathrm{u}$-cone fields of size $\theta$ as follows
\begin{equation}
\label{e.conefilds}
\begin{split}
\mathcal{C}^\mathrm s_{\theta}(p) & \eqdef
\Big\{(u, v, w)\in \mathbb{R}^3:
\sqrt{v^2 + w^2}<\theta|u|\Big\},
\\
\mathcal{C}^{\mathrm{uu}}_{\theta}(p)& \eqdef
\Big\{(u, v, w)\in\mathbb{R}^3:
\sqrt{u^2 + w^2}<\theta|v|\Big\},
\\
\mathcal{C}^{\mathrm u}_{\theta}(p) & \eqdef
\Big\{(u, v, w)\in\mathbb{R}^3:
|u|<\theta\sqrt{v^2 + w^2}\Big\}.
\end{split}
\end{equation}
Note that $\mathcal{C}^{\mathrm{uu}}_{\theta}(p)\subset 
\mathcal{C}^\mathrm u_{\theta}(p)$.

Related to these cone fields,  we define {\emph{$\mathrm{s}_\theta$-}} and {\emph{$\mathrm{uu}_\theta$-discs}} and 
{\emph{$\mathrm{u}_\theta$-strips}} as follows:
\begin{itemize}
\item
Let $L$ be a  regular curve. We say that $L$ is
 an \textit{$\mathrm{s}_\theta$-disc} if it is  contained in $\Delta$,
$T_pL\subset 
 \mathcal{C}^{\mathrm{s}}_{\theta}(p)$
for each $p\in L$, 
and its end-points are contained in different connected components $\partial^{\mathrm{s}}\Delta$.
Similarly, we say that $L$ is a
\textit{$\mathrm{uu}_\theta$-disc} if $L\subset \mathbb{R} \times I_y \times \mathbb{R}$, 
%\margem{Q1: why not $L\subset \Delta$?\textcolor{red}{En un comienzo asi lo teniamos, pero nos pusismos pesados con la definicion de blender-horseshoes (revisamos tu articulo con bonatti) y lo dejamos finalmente de esa forma...me parece q esta bien asi, podrian haber uu-discos (discos tangentes a los conos) que no esten contenidos en $\Delta$(esto va depender de la amplitud de los conos) por ejemplos uu-discos cerca de la frontera de la caja...me parece que definiendo  uu-discos de la forma en que est\'a, se pueden tomar ademas uu-discos en una ``vecindad'' de la caja... }}
$T_pL\subset 
\mathcal{C}^{\mathrm{uu}}_{\theta}(p)$ for each $p\in L$, and
its end-points are contained in different connected components 
of $\mathbb{R}\times \partial \mathrm{I}_y \times \mathbb{R}$.
\item
A surface $S\subset \Delta$
is a \textit{$\mathrm{u}_\theta$-strip}
if $T_pS \subset \mathcal{C}^\mathrm u_{\theta}(p)$ for every
$p$ in $S$ and there exists a $C^1$-embedding $E : \mathrm I_y\times \mathrm J\to\Delta$ (where $\mathrm J$ is a
subinterval of $\mathrm I_z$) such that 
$E(\mathrm I_y\times \mathrm J)=S$ and $L(z) \eqdef E(\mathrm I_y \times \{z\})$ is a $\mathrm{uu}_\theta$-disc
 for every
$z\in \mathrm J$. The {\em{width}} of $S$, denoted by $w(S)$,
is the infimum of the length
of the curves in $S$ which are transverse to 
$\mathcal{C}^{\mathrm {uu}}_{\theta}$ and
join the two 
components of $E(\mathrm I_y \times \partial \mathrm J)$
\end{itemize}

\begin{remark}[Right and left classes of $\mathrm{uu}$-discs]
\label{r.leftandright}
{\emph{
In what follows, we fix $\theta, \vartheta>0$. Note that every $\mathrm{s}_\vartheta$-disc $W$ such that $(W\setminus \partial W)$ is contained in the interior
of $\Delta$ defines two different (free)
homotopy classes of $\mathrm{uu}_\theta$-discs
  disjoint from $W$.
This allows us to consider $\mathrm{uu}_\theta$-discs \textit{at the left} and \textit{at the right} of $W$
(corresponding to the two different homotopy classes), denoted  by $\mathcal{U}^\ell_W$ and $\mathcal{U}^r_W$,
respectively.
The \textit{right class} $\mathcal{U}^r_W$ (resp., \textit{left class} $\mathcal{U}^\ell_W$) 
is the class containing the $\mathrm{uu}_\theta$-disc
 $\{0\} \times\mathrm I_y\times\{z^+\}$ (resp., containing the   $\{0\} \times\mathrm I_y\times\{z^-\}$). 
 With a slight abuse of notation, we also denote by 
 $\mathcal{U}^i_W$ the union of the $\mathrm{uu}$-discs in $\mathcal{U}^i_W$, $i=r,\ell$.}}
 
 {\emph{Similarly, a 
$\mathrm{u}$-strip $S$ through $\Delta$ is \textit{at the right}  (resp. \textit{at the left}) of
$W$ if it is foliated  by $\mathrm{uu}$-discs at the
right (resp. \textit{at the left}) of $W$.}}
\end{remark}
% 
%%If $L\in \mathrm{H}^r_W$ (resp. $L\in \mathrm{H}^\ell_W$) we say that $L$ is \textit{at the right}  (resp. \textit{at the left}) of $W$.
%Observe that if $W_1$ and $W_2$ are different horizontal curves in $\Delta$, then $\mathrm{H}^r_{W_1}\cap \mathrm{H}^\ell_{W_2}\neq \emptyset$
%or $\mathrm{H}^r_{W_2}\cap \mathrm{H}^\ell_{W_1}\neq \emptyset$.}
%

We are now ready to recall the definition of a blender-horseshoe in ~\cite{BDtang}.

\begin{defi}[Blender-Horseshoe] 
\label{d.BH}
{\em{
The maximal invariant $\Lambda_F\eqdef \cap_{i\in\mathbb{Z}}F^i(\Delta)\subset \mathrm{int}(\Delta)$
of a (local) diffeomorphism $F:\Delta\to F(\Delta) \subset \mathbb{R}^3$ is a \textit{blender-horseshoe}
 if   conditions (\BHa)-(\BHf) below hold:
\begin{itemize}
\item[(\BHa)]
{\emph{$\mathrm{s}$- and $\mathrm{u}$-legs }}:
 There are a connected subsets $\mathcal{A}$ and $\mathcal{B}$ of $\Delta$, 
 called {\emph{$\mathrm{s}$-legs of the blender}}, with
$$
\mathcal{A} \cap \mathcal{B} =\emptyset
\quad 
\mbox{and} 
\quad
(\mathcal{A} \cup \mathcal{B})\cap \partial
^{\mathrm{uu}}\Delta =\emptyset
$$
such that
$$
F(\Delta)\cap (\mathbb{R}\times \mathrm{I}_y\times\mathbb{R})
=
F(\mathcal{A}) \cup F(\mathcal{B})
 \subset (x^-,x^+) \times \mathrm{I}_y\times \mathbb{R} .
$$
\end{itemize}
Note that the sets 
$F(\mathcal{A})$ and $F(\mathcal{B})$ are the connected components of
$F(\Delta)\cap (\mathbb{R}\times \mathrm{I}_y\times\mathbb{R})$, they are called the
 {\emph{$\mathrm{u}$-legs of the blender}}. See Figure \ref{fig:legs}
%\quad\mbox{and}\quad
%\textcolor{magenta}{
%(\mathcal{A} \cup \mathcal{B}) \cap \partial^
%{\mathrm{uu}}\Delta = \emptyset.}
%$$
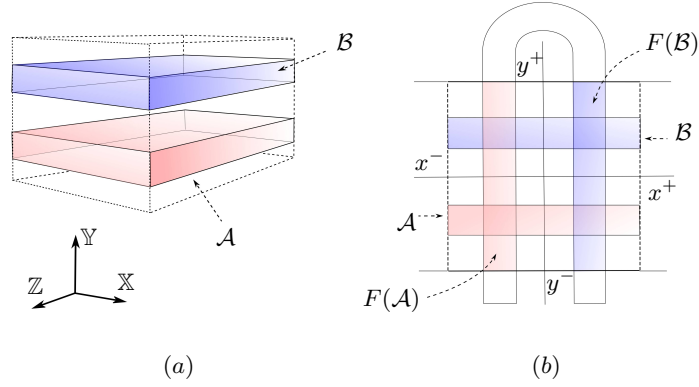
\begin{figure}
\centering
\begin{overpic}[scale=.20,bb=0 0 1418 710,tics=5
]{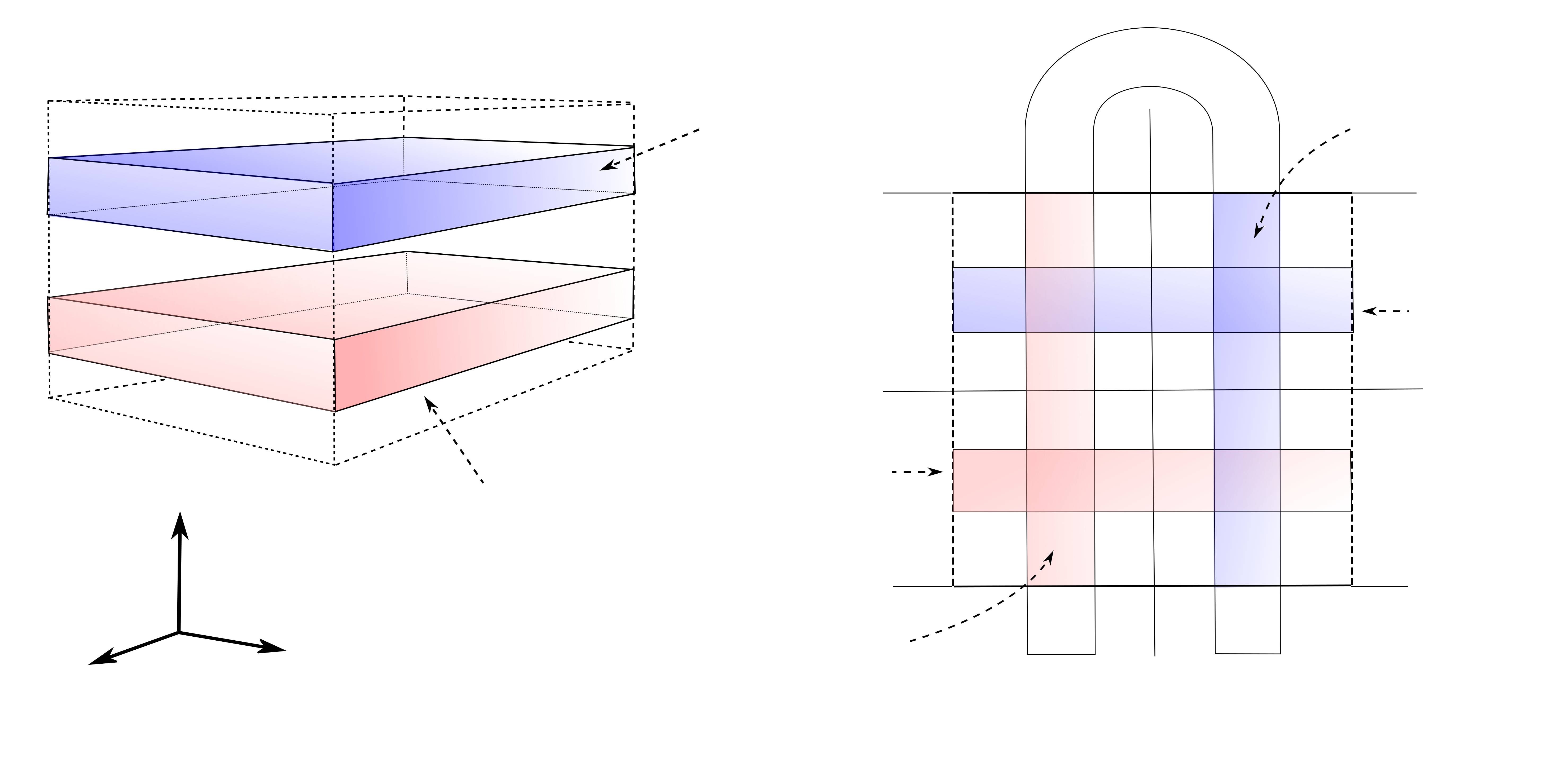}
   \put(23,-1){\small{$(a)$}} 
     \put(45,42){\small{ $\mathcal{B}$}}
      \put(29,16){\small{ $\mathcal{A}$}}
      \put(12,16){\small{$\mathbb{Y}$}}
        \put(17,10){\small{$\mathbb{X}$}}
        \put(5,9.5){\small{$\mathbb{Z}$}}
        \put(72,-1){\small{$(b)$}} 
        \put(53,18){\small{ $\mathcal{A}$}}
         \put(48.5,8){\small{ $F(\mathcal{A})$}}
          \put(90,30){\small{ $\mathcal{B}$}}
         \put(86,42){\small{ $F(\mathcal{B})$}}
          \put(87.5,22){\small{${x}^+$}}
           \put(56.5,26){\small{$x^-$}}
            \put(70,39){\small{$y^+$}}
           \put(74,10){\small{$y^-$}}
   \end{overpic}
\caption{$(a)$ s-legs of the blender-horseshoes. $(b)$ Projection of $F(\Delta)\cap (\mathbb{R}\times \mathrm{I}_y\times\mathbb{R})$ in the plane  $\mathbb{XY}$. }
\label{fig:legs}
%\vspace{0.2cm}
\end{figure}
\begin{itemize}
\item[(\BHb)]{\emph{Contracting and expanding  invariant cone fields}}.
There exist $\theta,\vartheta>0$, $\ell\in\mathbb{N}$, $c>1$,
and cone fields $\mathcal{C}_{\vartheta}^{\mathrm{s}}$, 
 $\mathcal{C}_{\theta}^{\mathrm{u}}$,
 and  $\mathcal{C}_{\theta}^{\mathrm{uu}}$
 such that:
\begin{itemize}
\item[(i)] \textit{Strict invariance}:  for every
$p\in \mathcal{A}\cup \mathcal{B}$ we have that
\[
\begin{split}
&DF^\ell_{p}(\mathcal{C}_{\vartheta}^{\mathrm{s}}(p))\supset \mathcal{C}_{\vartheta}^{\mathrm{s}}(F^\ell(p)),
\\
&DF^\ell_{p}(\mathcal{C}_{\theta}^{\mathrm{u}}(p))\subset \mathcal{C}_{\theta}^{\mathrm{u}}(F^\ell(p)),
\,\, \mbox{and} \,\,
DF^\ell_{p}(\mathcal{C}_{\theta}^{\mathrm{uu}}(p))\subset \mathcal{C}_{\theta}^{\mathrm{uu}}(F^\ell(p)).
\end{split}
\]
\item[(ii)]
\textit{Expansion/Contraction.}
For every $v\in \mathcal{C}_{\vartheta}^{\mathrm{s}}(p)$ and 
every $w\in \mathcal{C}_{\theta}^{\mathrm{u}}(p)$ 
we have that
$$
|DF^\ell_pv|\le c^{-1}|v| \,\, \mbox{and} \,\, |DF^\ell_pw|\ge c|w|.
$$
\end{itemize}
\end{itemize}
Conditions (\BHa) and (\BHb) imply the existence of two fixed saddles $P\in \mathcal{A}$ and 
$Q\in \mathcal{B}$, called the
 \textit{reference saddles} of $\Lambda_F$. We define the local stable manifolds of $P$ and $Q$ by 
\begin{equation}\label{e.WW}
W^{\mathrm s}_{\mathrm {loc}}(R)\eqdef
\mbox{connected component of $W^{\mathrm s}(R)\cap \Delta$ containing $R$},
\end{equation}
 where $R=P,Q$.
%$$
%W^{\mathrm s}_{\mathrm {loc}}(Q)\eqdef
%\mbox{connected componet of $W^{\mathrm s}(Q)\cap \Delta$ containig $Q$}.
%$$
These local stable manifolds 
are $\mathrm{s}$-discs (in what follows we omit the dependence of $\theta$ and $\vartheta$).
Thus, either
$\mathcal{U}^\ell_{W^{\mathrm s}_{\mathrm {loc}}(P)}\cap \mathcal{U}^r_{W^{\mathrm s}_{\mathrm {loc}}(Q)}\neq \emptyset$ or
 $\mathcal{U}^r_{W^{\mathrm s}_{\mathrm {loc}}(P)}\cap \mathcal{U}^\ell_{W^{\mathrm s}_{\mathrm {loc}}(Q)}\neq \emptyset$. We assume that the first case holds
and denote by
$
\mathcal{U}^b \eqdef \mathcal{U}^\ell_{W^{\mathrm s}_{\mathrm {loc}}(P)}
\cap
\mathcal{U}^r_{W^{\mathrm s}_{\mathrm {loc}}(Q)}.
$
The family of discs $\mathcal{U}^b$ is called the {\emph{superposition region}} 
of the blender-horseshoe.
We say that a $\mathrm{uu}$-disc is {\emph{in between}} if it is contained 
$\mathcal{U}^b$. Similarly, a  $\mathrm{u}$-strip is {\emph{in between}} 
if it is foliated by $\mathrm{uu}$-discs in between.
%
%$\mathrm{H}^r_{W^{\mathrm s}_{\mathrm {loc}}(P)}\cap
%\mathrm{H}^\ell_{W^{\mathrm s}_{\mathrm {loc}}(Q)}$. 

\begin{itemize}
\item[(\BHc)]
{\emph{Markov partition.}} 
The connected components of $F^{-1}(\Delta)\cap\Delta$ are the sets
$$
\mathbb{A}\eqdef F^{-1}(F(\mathcal{A})\cap \Delta)\quad\mbox{and}\quad \mathbb{B}\eqdef F^{-1}(F(\mathcal{B})\cap \Delta),
$$
which satisfy
\[
\begin{split}
\mathbb{A}\cup \mathbb{B}\subset \mathrm{I}_x \times (y^-,y^+) \times (z^-,z^+),\quad 
F(\mathbb{A}) \cup F(\mathbb{B})
 \subset (x^-,x^+) \times \mathrm{I}_y\times \mathbb{R} .
\end{split}
\]
\end{itemize}

\begin{itemize}
\item[(\BHd)]
\textit{$\mathrm{uu}$-discs through the local stable manifolds of $P$ and $Q$}:
 Let $L$ and $L' $ be $\mathrm{uu}$-discs
such that $L \cap W^\mathrm{s}_\mathrm{loc}(P)\neq\emptyset$ 
and 
$L'\cap W^\mathrm{s}
_\mathrm{loc}(Q)\neq\emptyset.
$
Then
\begin{equation*}
 L\cap
\overline{  
  \big(\partial^\mathrm{u}\Delta\setminus
\partial^\mathrm{uu}\Delta\big)}=\emptyset  
  ,\quad 
  L'\cap
  \overline{ \big(\partial^\mathrm{u}\Delta\setminus
\partial^\mathrm{uu}\Delta\big)}=\emptyset.
\end{equation*}
\item[(\BHe)]
\textit{Positions of images of $\mathrm{uu}$-discs}:
Let $L$ be a $\mathrm{uu}$-disc in $\Delta$ and consider
$$
L_\mathcal{C}\eqdef L\cap \mathcal{C}, \quad \mathcal{C}=\mathcal{A}, \mathcal{B}.
$$ 
By $(\BHa)$ and $(\BHb)$, $F(L_{\mathcal{C}})$ is a uu-discs in $\mathrm{I}_x\times \mathrm{I}_y\times\mathbb{R}$.
The relative position
of $F(L_{\mathcal{C}})$ obeys the following rules:
\begin{itemize}
\item [(1)]
if $L\in \mathcal{U}^r_{W^{\mathrm s}_{\mathrm {loc}}(P)}$ then 
$F(L_\mathcal{A})\in \mathcal{U}^r_{W^{\mathrm s}_{\mathrm {loc}}(P)}$,
\item [(2)]
if $L\in \mathcal{U}^\ell_{W^{\mathrm s}_{\mathrm {loc}}(P)}$ then 
$F(L_\mathcal{A})\in \mathcal{U}^\ell_{W^{\mathrm s}_{\mathrm {loc}}(P)}$,
\item [(3)]
 if $L\in \mathcal{U}^r_{W^{\mathrm s}_{\mathrm {loc}}(Q)} $
 then $F(L_\mathcal{B})\in \mathcal{U}^r_{W^{\mathrm s}_{\mathrm {loc}}(Q)}$,
 \item [(4)]
if $L\in \mathcal{U}^\ell_{W^{\mathrm s}_{\mathrm {loc}}(Q)}$ then 
$F(L_\mathcal{B})\in \mathcal{U}^\ell_{W^{\mathrm s}_{\mathrm {loc}}(Q)}$,
 \item [(5)]
if $L\in \mathcal{U}^r_{W^{\mathrm s}_{\mathrm {loc}}(P)}$ 
or  $L\cap W^\mathrm{s}_\mathrm{loc}(P)\neq\emptyset$
 then $F(L_\mathcal{B})\in \mathcal{U}^r_{W^{\mathrm s}_{\mathrm {loc}}(P)}$, and
\item [(6)]
 if $L\in \mathcal{U}^\ell_{W^{\mathrm s}_{\mathrm {loc}}(Q)}$
or  $L\cap W^\mathrm{s}_\mathrm{loc}(Q)\neq\emptyset$
 then $F(L_\mathcal{A})\in \mathcal{U}^\ell_{W^{\mathrm s}_{\mathrm {loc}}(Q)}$. 
 \end{itemize}
\end{itemize} 
\begin{itemize} 
\item[(\BHf)]
 {\textit{Positions of images of $\mathrm{uu}$-discs in  $\mathcal{U}^b$}}:
 Let $L$ be a $\mathrm{uu}$-disc in $\Delta$ such that $L\in\mathcal{U}^b$,
 then either $F(L_\mathcal{A})$
 or $F(L_\mathcal{B})$ is contained in
 $\mathcal{U}^b$. 
\end{itemize}
}}
 \end{defi}
Figure~\ref{fig:Prot} illustrates a prototypical blender-horseshoe. 
%\margem{Q5: bueno, la figura esta muy bien, pero creo que ilustra poco las condiciones, creo que es depicts\textcolor{red}{Retoqu\'e la figura, y creo que el caption es aclaratorio....ademas quen se sienta a leer la definicion de b-h, se dar\'a cuenta que la deficion  asegura una relacion precisa entre la frontera del cubo y su imagen, esto est\'a perfectamente ilustrado en el dibujo.... creo que no hay mas que agregar... }} 

We now pointed out some consequences of conditions (\BHa)-(\BHf), see \cite[ Section 3.2.4]{BDtang} for more details.

\begin{remark}{\em{
 \label{r.golpedeestado}
$\,$
\begin{itemize}
\item
The existence of the invariant (contracting or expanding) cone fields in (\BHb) implies
the hyperbolicity (and partial hyperbolicity) of the set $\Lambda_F$:
the set $\Lambda_F$ is hyperbolic and  partially hyperbolic with a dominated splitting
 $$
 T_{\Lambda_F} (\mathbb{R}^3) =
 E^\mathrm{s} \oplus  E^\mathrm{cu} \oplus  E^\mathrm{uu},
 $$
 where $E^\mathrm{s}$ and  $E^\mathrm{u}= E^\mathrm{cu} \oplus  E^\mathrm{uu}$ 
 are the stable and
  unstable bundles of $\Lambda_F$, respectively.
\item
From (\BHa)-(\BHb), one gets that $\{\mathbb{A},\mathbb{B}\}$ is a Markov partition generating
$\Lambda_F$.
 Therefore,  the dynamics of $F$ in $\Lambda_F$ is  hyperbolic and conjugate to the full shift of two symbols. In particular,
the set $\Lambda_F$ contains exactly two fixed points of $F$, 
$P \in \mathbb{A}$ and $Q\in \mathbb{B}$.
\item 
Since $\Lambda_F$ is locally maximal, we have that 
%\margem{Q7: esto lo he cambiado de lugar para evitar repeticiones \textcolor{red}{me parece ok!}}
$$
W^{\mathrm s}_{\mathrm{loc}}(\Lambda_F)
\eqdef
\bigcap_{n\in \mathbb{N}}
F^{-n} (\Delta) =
\bigcup_{x\in \Lambda_F}W^{\mathrm s}_{\mathrm{loc}}(x)\subset 
W^{\mathrm{s}} (\Lambda_F),
$$
where $W^{\mathrm s}_{\mathrm{loc}}(x)$ is the  connected component of $W^{\mathrm s}(x)\cap \Delta$ containing $x$. We can  write the local stable manifold  $W^{\mathrm s}_{\mathrm{loc}}(\Lambda_F)$ as the Cartesian product of a Cantor set, say $C$, by an interval. This Cantor set is ``fat" in the following sense:  the projection of $C$ in  the center-unstable direction contains (open) intervals. See Figure~\ref{fig:Prot}-(b).
\item
Conditions (\BHa)-(\BHf) are $C^1$-open.
 Hence if $\Lambda_F$ is a blender-horseshoe of $F$ then  the 
continuation $\Lambda_G$ of $\Lambda_F$ is a blender-horseshoe  for every $G$ sufficiently
  $C^1$-close to $F$
(with the same reference cube $\Delta$).
\end{itemize}
}}
\end{remark}  

\begin{figure}
\centering
\begin{overpic}[scale=.20,bb=0 0 1095 834,tics=5
  ]{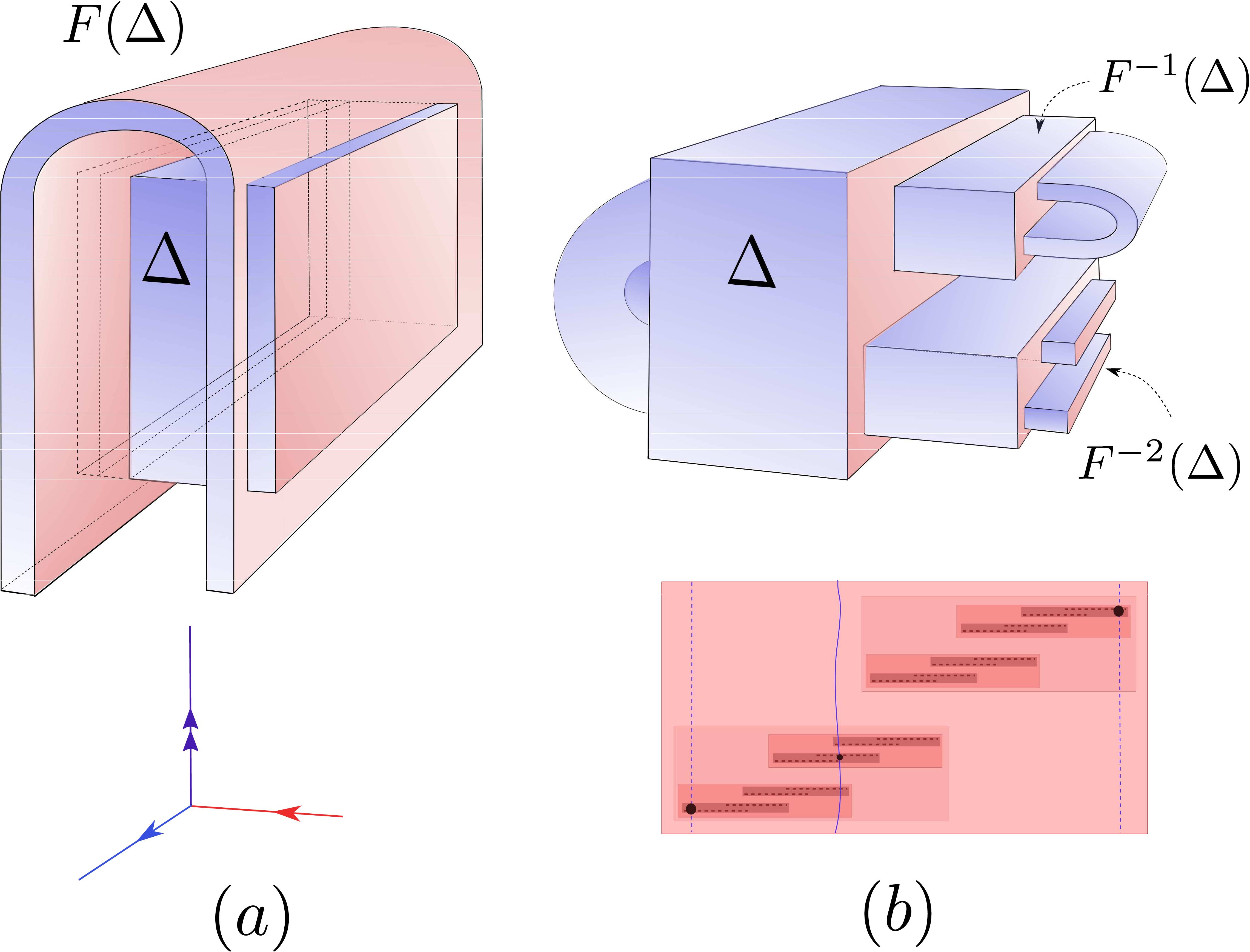}
            \put(48,10){\small{$Q$}}
           \put(92.5,26){\small{$P$}}
            \put(62.5,22){\small{$L$}}
     \put(25,13){\small{$\mathbb{X}$}}
           \put(16.5,22){\small{$\mathbb{Y}$}}
            \put(6,8.5){\small{$\mathbb{Z}$}}       
    \end{overpic}
\caption{$(a)$ Prototypical blender-horseshoe. $(b)$ Projection in plane $\mathbb{YZ}$ of a $\mathrm{uu}$-disc $L$ in the region of superposition of the blender-horseshoe $
\mathcal{U}^b$. }
\label{fig:Prot}
%\vspace{0.2cm}
\end{figure}
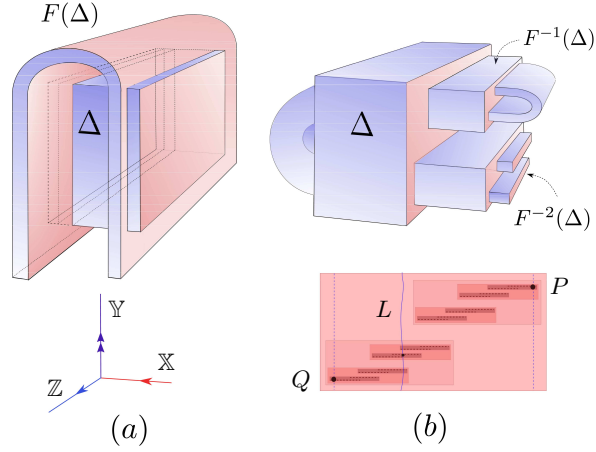

The next lemma states the
  {\emph{distinctive property}} of a blender-horseshoe.
  
   \begin{lemma}[Lemma 3.13 in \cite{BocBonDia:16}]\label{l.sp}
 For every
 $L\in \mathcal{U}^b$ it holds
 $L \cap W^{\mathrm{s}}_{\mathrm{loc}} (\Lambda_F)\ne\emptyset$.
 \end{lemma}
 
 \begin{proof}
 Consider $L=L_0'\in \mathcal{U}^b$. By condition (BH6), $F(L)$ contains a disc $L_1'\in \mathcal{U}^b$. We let
 $F^{-1}(L_1')=L_1\subset L$. We  inductively define $L_n\subset L$ and $L_n'\in \mathcal{U}^b$ for 
 $n>1$  as follows. 
 Assuming defined $L_{n-1}'\in \mathcal{U}^b$ and 
 $L_{n-1}\subset L_0$
 with $L_{n-1}'\subset 
 F(L_{n-2}')$ and $F^{-n+1}(L_{n-1}') =L_{n-1}$, we consider $L_{n}'\in \mathcal{U}^b$ contained in $F(L_{n-1}')$ and let $F^{-n}(L_n')=L_n\subset L$. The sequence $(L_n)$ is nested and hence $\emptyset\ne \bigcap_n L_n\subset L$. By construction, $\bigcap_n L_n\subset W^{\mathrm{s}}_{\mathrm{loc}} (\Lambda_F)$.
 \end{proof}
 
 We also have the following refinement  of the above lemma.

\begin{lemma}
\label{l.dyn-geo}
Every $\mathrm{u}$-strip 
in between
 intersects transversely  $W^\mathrm s
(P)$.
 \end{lemma}
\begin{proof}
Note that $F^{-1} (W^\mathrm{s}_{\mathrm{loc}}(P) )\cap \Delta$ consists of
two connected components. We denote by $W^{\mathrm s}_0$ the connected component
that does not contain $P$. Note that this set is an $\mathrm{s}$-disc. Observe that there is $\alpha>0$ such that every $\mathrm{u}$-strip $S$ with
$w(S)>\alpha$  intersets $W^{\mathrm s}_0$ transversely.
%Since $\Lambda_F$ is the maximal
%invariant set in $\Delta$, then
%$W^{\mathrm s}_{\mathrm{loc}}(\Lambda_F)=\cap_{i\in\mathbb{N}}F^{-i}(\Delta)$.
%Thus, it is sufficient to see that the stable manifold $W^{\mathrm s}(P^{*},F)$ of $P^{*}\in \Lambda_F$ intersects transversally
%every vertical strip $S$ through $\Delta$
%to
%the right of $W^{\mathrm s}_0$
%(note that any vertical segment $D$ through $\Delta$ can be seen as an intersection of a nested sequence of vertical strip $S$ throughout $\Delta$).
%To see that, 
Conditions (\BHb) and (\BHf)
imply that the width of a $\mathrm{u}$-strip $S\subset \Delta$ in between grows exponentially after iterations by $F$
(for simplicity let us assume that $\ell$ in (\BHb) is $\ell=1$):
there is $c'>1$ (independent of the strip) such that
there are two possibilities, either $F(S)$ intersects (transversely) 
$W^\mathrm{s}_{\mathrm{loc}}(P) $ or 
$F(S)$
contains a $\mathrm{u}$-strip $S'$ in between  such that  $w (S') > c'w(S)$.

Take now a $\mathrm{u}$-strip $S=S_0$ in between. If $S\cap W^{\mathrm s}_0\ne \emptyset$ we are done.
Otherwise we consider $F(S)$. If $F(S)$ intersects either  $W^{\mathrm s}_0$ or
$W^\mathrm{s}_{\mathrm{loc}}(P)$ we are also done. Otherwise we get 
a new $\mathrm{u}$-strip $S_1$ in between contained in $F(S_0)$ with $w(S_1)> c'w(S_0)$. 
We now argue inductively, at some step we get a first $n$ such that either $F(S_n)$ intersects
$W^{\mathrm s}_0$ or
$W^\mathrm{s}_{\mathrm{loc}}(P)$ or $w(S_n)>\alpha$ and hence $S_n$ intersects $W^\mathrm{s}_0$.
In both cases, we are done.
This proves the lemma.
\end{proof}

 \subsubsection{Blender-horseshoes for endomorphisms}
 \label{ss.Blender-horseshoesforendomorphisms}
For endomorphisms the blender hor\-se\-shoe are defined as in the case of  diffeomorphisms.
\begin{defi}[Blender-horseshoes for endomorphisms]
\label{d.BHtoEndo}
{\em{ 
The maximal invariant set ${\Lambda}_G:=
\bigcap_{i\in\mathbb{Z}} G^i(\Delta) \subset \mathrm{int}(\Delta)$ 
of an endomorphism $G:\Delta\to\mathbb{R}^3$ is a \textit{blender-horseshoes}
 if $G$ satisfies the conditions (\BHa)-(\BHf). 
}}
\end{defi}
 
\begin{remark}[Continuations of blender-horseshoes for endomorphisms]
\label{r.electionday}
{\em{
Assume\newline that the endomorphism $G$ has a blender-horseshoe in $\Delta$.
Then every diffeomorphism or endomorphism
$F$ such that $F|_\Delta$  is sufficiently close to $G|_{\Delta}$ has a blender-horseshoe in $\Delta$.
}}
\end{remark} 

\section{Proof of Theorem~\ref{t.BH-DKS}}
\label{s.localperturbationsgk}
Theorem~\ref{t.BH-DKS} is  a consequence of following result and Remark~{r.electionday}.

\begin{teo}
\label{t.BH}
For every $(\xi,\mu)\in \mathcal{P}\eqdef (1.18,1.19)\times(-10,-9)$, the en\-do\-morphism
$$
G_{(\xi,\mu,0,0)}(x,y,z)=(y,\mu+y^2,\xi\,z+y) 
$$
has a blender-horseshoe in $\Delta=[-4,4]^2\times[-40,22]$.
\end{teo}

The proof of this theorem involves some preliminary steps. 
First, for  the endomorphisms 
$G_{\xi,\mu}\eqdef G_{(\xi,\mu,0,0)}$, where
$ (\xi,\mu)\in \mathcal{P}$,
we study their hyperbolic fixed points
and their invariant manifolds. 
As we will see, these fixed points will be the reference saddles of the blender-horseshoe of $G_{\xi,\mu}$ in $\Delta$.

\subsection{Hyperbolic fixed points of $G_{\xi,\mu}$}
\label{ss.propiedades}
We calculate the hyperbolic fixed points of $G_{\xi,\mu}$ and their invariant manifolds. %\begin{equation}\label{e.setPofparameters}
%\mathcal{P}:=(1.18,1.19)\times(-10,-9).
%\end{equation} 
\begin{lemma}
\label{l.Hyperbolic-fixed-points}
For every $(\xi,\mu)\in \mathcal{P}$, the endomorphism $G_{\xi,\mu}$ has two hyperbolic fixed saddles 
%$P^\pm_{\xi,\mu}=(x^\pm_{\xi,\mu}, y^\pm_{\xi,\mu}, z^\mp_{\xi,\mu})\in  \Delta$, 
%where \margem{notacion??}
%\begin{equation}
%\label{e.defpuntofijo}
%\begin{split}
%&x^\pm_{\xi,\mu}= y^\pm_{\xi,\mu} = \mu + (y^\pm_{\xi,\mu})^2=
% (1-\xi)\,z^\mp_{\xi,\mu}, \\
%&y^\pm_{\xi,\mu}=y^\pm_{\mu}:=\frac{1\pm(1-4\,\mu)^{1/2}}{2}.
%\end{split}
%\end{equation}
$P_{\xi,\mu}=(p_{\xi,\mu}, p_{\xi,\mu}, \tilde p_{\xi,\mu})$ and 
$Q_{\xi,\mu}=(q_{\xi,\mu}, q_{\xi,\mu}, \tilde q_{\xi,\mu})$
in  $\Delta$, where
\begin{equation}
\label{e.defpuntofijo}
\begin{split}
& p_{\xi,\mu} = \mu + (p_{\xi,\mu})^2=
 (1-\xi)\,\tilde p_{\xi,\mu},\quad p_{\xi,\mu}=p_{\mu}=\frac{1-(1-4\,\mu)^{1/2}}{2},
 \\
&q_{\xi,\mu} = \mu + (q_{\xi,\mu})^2=
 (1-\xi)\,\tilde q_{\xi,\mu},\quad q_{\xi,\mu}=q_{\mu}=\frac{1+(1-4\,\mu)^{1/2}}{2}.
\end{split}
\end{equation}
\end{lemma}
\begin{proof}
A simple calculation shows that 
$P_{\xi,\mu}=(p_{\mu}, p_{\mu}, \tilde p_{\xi,\mu})$ and 
$Q_{\xi,\mu}=(q_{\mu}, q_{\mu}, \tilde q_{\xi,\mu})$
are the two solutions 
of $G_{\xi,\mu}(x,y,z)=(x,y,z)$.
Using equation \eqref{e.defpuntofijo} and that 
$(\xi,\mu)\in\mathcal{P}$, we get the following estimates for the coordinates of $P_{\xi,\mu}$ and $Q_{\xi,\mu}$:
\begin{equation}\label{e.est}
\begin{split}
-2.7&<
 %y^-_{\mu}
p_\mu  <-2.5, 
\quad
13<
\tilde p_{\xi,\mu}
% z^+_{\xi,\mu} 
< 15,
\\
3.5&<
q_\mu
%y^+_{\mu}
< 3.71,
\quad
-20.6<
 \tilde q_{\xi,\mu}
 % z^-_{\xi,\mu}
 < -18.4.
 \end{split}
\end{equation}
Thus, $P_{\xi,\mu}, Q_{\xi,\mu}\in \Delta$.
We observe that  the eigenvalues of  $DG_{\xi,\mu}(P_{\xi,\mu})$, and $DG_{\xi,\mu}(Q_{\xi,\mu})$ are, respectively,
\begin{equation*}
\begin{split}
&\lambda^{\mathrm{s}}(P_{\xi,\mu})=0,
\quad \lambda^\mathrm{cu}(P_{\xi,\mu})=\xi,\quad
 \lambda^\mathrm{uu}(P_{\xi,\mu})=
 2\, p_{\mu},
 \\
&\lambda^{\mathrm{s}}(Q_{\xi,\mu})=0,
\quad \lambda^\mathrm{cu}(Q_{\xi,\mu})=\xi,\quad
 \lambda^\mathrm{uu}(Q_{\xi,\mu})=2\,q_{\mu},
\end{split}
\end{equation*}
 with respective eigenvectors
\begin{equation*}
\begin{split}
&v^{\mathrm s}(P_{\xi,\mu})=(1,0,0), \,\, v^{\mathrm{cu}}(P_{\xi,\mu})=(0,0,1), \,\,
v^{\mathrm{uu}}(P_{\xi,\mu})=\big(2\,p_{\mu}-\xi, 2\,p_{\mu}(\,p_{\mu}-\xi), 2\,p_{\mu}\big),
\\
&v^{\mathrm s}(Q_{\xi,\mu})=(1,0,0), \,\, v^{\mathrm{cu}}(Q_{\xi,\mu})=(0,0,1), \,\,
v^{\mathrm{uu}}(Q_{\xi,\mu})=\big(2\,q_{\mu}-\xi, 2\,q_{\mu}\,(q_{\mu}-\xi), 2\,q_{\mu}\big).
\end{split}
\end{equation*} 
As $\xi>1$ and
$|\lambda^{\mathrm{uu}}(P_{\xi,\mu})|=2\,|\,p_{\mu}|>5$ and $|\lambda^{\mathrm{uu}}(Q_{\xi,\mu})|=2\,|\,q_{\mu}|>7$,
we have that $P_{\xi,\mu}$ and  $Q_{\xi,\mu}$ are hyperbolic fixed points of 
$G_{\xi,\mu}$ for every $(\xi,\mu)\in \mathcal{P}$,
 ending the proof of the lemma.
\end{proof}

\begin{remark}[Invariant directions and foliations]
\label{r.Invariant directions}
{\em{For $R=P,Q$ consider the eigen\-spaces 
\begin{equation*}
\begin{split}
E^{\mathrm s}(R_{\xi,\mu})&\eqdef\mathbb{R}\times\{(0,0)\}\quad\mbox{and}\quad
 E^{\mathrm {cu}}(R_{\xi,\mu})\eqdef\{(0,0)\}\times
 \mathbb{R},
\end{split}
\end{equation*}
associated to the eigenvalues
 $\lambda^{\mathrm{s}}(R_{\xi,\mu})=0$ and 
  $\lambda^\mathrm{cu}(R_{\xi,\mu})=\xi>1$, and consider the straight lines  through $R_{\xi,\mu}$:
$$
\big\{R_{\xi,\mu}
+(t,0,0)
:t\in\mathbb{R}\big\}
\quad\mbox{and}\quad
 \big\{R_{\xi,\mu}
+(0,0,t): t\in\mathbb{R} \big\}.
$$
These lines are, respectively, 
tangent to the eigenspaces
$E^{\mathrm s}(R_{\xi,\mu})$ and 
$E^{\mathrm {cu}}(R_{\xi,\mu})$ at $R_{\xi,\mu}$, 
 and invariant by $G_{\xi,\mu}$:
\begin{equation*}
\begin{split}
&G_{\xi,\mu}\big(R_{\xi,\mu}+(t,0,0)\big)
=R_{\xi,\mu},\quad
G_{\xi,\mu}\big(R_{\xi,\mu}+(0,0,t)\big)
=R_{\xi,\mu}+(0,0,\xi t),
\end{split}
\end{equation*}
for every $t\in\mathbb{R}$.
%We define
%\begin{equation}
%\begin{split}
%\label{e.Wc}
%W^{\mathrm{c},\pm}_{\xi,\mu}(\Delta)
%\eqdef \Big\{(y^{\pm}_{\mu},y^{\pm}_{\mu},
%z^{\mp}_{\xi,\mu}+t)
%:t\in\mathbb{R}\Big\}\cap \Delta,
%\\
%W^{\mathrm{s},\pm}_{\xi,\mu}(\Delta)
%\eqdef \Big\{(y^{\pm}_{\mu}+t,y^{\pm}_{\mu},
%z^{\mp}_{\xi,\mu})
%:t\in\mathbb{R}\Big\}\cap \Delta.
%\end{split}
%\end{equation}
Moreover, 
\begin{equation}
\label{e.W}
W^{\mathrm{s}}(R_{\xi,\mu})=\big\{R_{\xi,\mu}
+(t,0,0)
:t\in\mathbb{R}\big\}, \quad R=P,Q.
\end{equation}
We define the \textit{center unstable manifold} of $R_{\xi,\mu}$ by
\begin{equation}\label{e.manifolds}
{W}^{\mathrm{cu}}(R_{\xi,\mu})\eqdef \big\{R_{\xi,\mu}
+(0,0,t): t\in\mathbb{R} \big\}, \quad R=P,Q.
\end{equation}

Consider  the endomorphism of $\mathbb{R}^2$ obtained by
 projecting  $G_{\xi,\mu}$ 
into the $\mathbb{YZ}$-plane,
\begin{equation}
\label{e.g}
g_{\xi,\mu}:\mathbb{R}^2\to\mathbb{R}^2,\quad g_{\xi,\mu}(y,z)\eqdef (\mu+y^2,\xi\,z+y).
\end{equation}
This endomorphism
 preserves the foliation
$\mathcal{F}=\big\{\{y\}\times \mathbb{R}:y\in \mathbb{R} \big\}$.
In particular, for $r=p,q$, the leaves
\begin{equation*}
\begin{split}
%\label{e.wcplus}
{W}^{\mathrm{cu}}_{\xi,\mu}(r_{\mu},\tilde r_{\xi,\mu})\eqdef
\big\{(r_{\mu},\tilde r_{\xi,\mu}+t):t\in \mathbb{R}
\big\},
\end{split}
\end{equation*}
are invariant by $g_{\xi,\mu}$.
}}
\end{remark}

\subsection{The legs of the blender-horseshoe}\label{ss.legs}
In this section, we will concentrate on property (\BHa) of  blender-horseshoes.
The definitions of   $\mathrm{s}$- and $\mathrm{u}$-legs  involve some preliminary
 constructions that we describe below.
 
For $\mu\in(-10,-9)$, consider the
points
\begin{equation}
\begin{split}
\label{e.ys}
&a_{\mu}\eqdef -\sqrt{4-\mu},\quad b_{\mu}\eqdef -\sqrt{-4-\mu},\quad c_{\mu}\eqdef \sqrt{-4-\mu},\quad d_{\mu}\eqdef \sqrt{4-\mu}.
\end{split}
\end{equation}
Note that if $\mu\in(-10,-9)$
it holds
\begin{equation}
\label{e.estim} 
\begin{split}
-\sqrt{14}&<a_{\mu}=-d_{\mu}<
-\sqrt{13},
\quad
-\sqrt{6}<b_{\mu}=-c_{\mu}<
-\sqrt{5}.
%\\
%\sqrt{5}&<c_{\mu}<
%\sqrt{6},\quad 
%\sqrt{13}
%<d_{\mu}<\sqrt{14}.
\end{split}
\end{equation}
Consider the intervals
$\I_{\mu}\eqdef[a_{\mu},b_{\mu}]$ and $\J_{\mu}\eqdef[c_{\mu},d_{\mu}]$. The choice of the parameter $\mu$ and the estimates in \eqref{e.estim} imply that
\begin{equation}\label{e.valoresy}
\I_{\mu}=[a_{\mu},b_{\mu}]\subset (-4,0)\quad\mbox{and}\quad \J_{\mu}=[c_{\mu},d_{\mu}]\subset (0,4).
\end{equation}
Consider the 
sub-cubes of $\Delta$ defined by
\begin{equation}\label{e.AmBm}
\begin{split}
\mathcal{A}_{\xi,\mu}\eqdef[-4,4]\times \I_{\mu}\times[-40,22],\quad 
\mathcal{B}_{\xi,\mu}\eqdef[-4,4]\times \J_{\mu}\times[-40,22].
\end{split}
\end{equation}
From \eqref{e.valoresy} it follows
$$
\mathcal{A}_{\xi,\mu}\cap \mathcal{B}_{\xi,\mu}=\emptyset
\quad
\mbox{and}
\quad
(\mathcal{A}_{\xi,\mu}\cup \mathcal{B}_{\xi,\mu})\cap \partial^{\mathrm{uu}} \Delta=\emptyset.
$$

\begin{remark}
\label{r.fixpoint}
{\em{If $\mu\in (-10,-9)$ then
 $p_{\mu}\in(a_{\mu},b_{\mu})$,
$q_{\mu}\in(c_{\mu},d_{\mu})$, and
 thus $P_{\xi,\mu}\in \mathrm{interior}(\mathcal{A}_{\xi,\mu})$ and $Q_{\xi,\mu}\in \mathrm{interior}(\mathcal{B}_{\xi,\mu})$. 
}}
\end{remark}

Hence the sets  $\mathcal{A}_{\xi,\mu}$ and $\mathcal{B}_{\xi,\mu}$ satisfy the first part of condition (\BHa).
To prove that 
$G_{\xi,\mu}(\mathcal{A}_{\xi,\mu})$ and 
$G_{\xi,\mu}(\mathcal{B}_{\xi,\mu})$ satisfy the second  
part  of (\BHa),
as in the case of the boundary of $\Delta$, we split  the boundary of 
$\mathcal{A}_{\xi,\mu}$  as follows. Let 
%(see Figure~\ref{fig:Split}), let
\begin{equation*}
\begin{split}
\partial^{\mathrm{uu}} \mathcal{A}_{\xi,\mu}&
\eqdef [-4,4]\times 
\partial \I_{\mu}\times [-40,22],\\
\partial^{\mathrm{u}} 
\mathcal{A}_{\xi,\mu}&\eqdef [-4,4]\times
 \partial\big(\I_{\mu}
 \times[-40,22]\big),\\
\partial^{\mathrm{s}} 
\mathcal{A}_{\xi,\mu}&\eqdef \partial([-4,4])
\times \I_{\mu}\times[-40,22].
\end{split}
\end{equation*}
Note that 
$\partial \mathcal{A}_{\xi,\mu}=
 \partial^{\mathrm{u}} \mathcal{A}_{\xi,\mu}
\cup \partial^{\mathrm{s}} \mathcal{A}_{\xi,\mu}$ and $\partial^{\mathrm{uu}} 
\mathcal{A}_{\xi,\mu}\subset \partial^{\mathrm{u}} 
\mathcal{A}_{\xi,\mu}.$
Analogously, we split the boundary of  $\mathcal{B}_{\xi,\mu}$.

\begin{remark}\label{r.front}
{\emph{
We observe that for $\mathcal{C}=\mathcal{A},\mathcal{B}$ it holds that
%\margem{Q9: L \textcolor{red}{front and rear conver???}}
\begin{equation*}
\begin{split}
\overline{
\partial \mathcal{C}_{\xi,\mu}\setminus
(\partial^{\mathrm{uu}} \mathcal{C}_{\xi,\mu}\cup 
\partial^{\mathrm{s}} \mathcal{C}_{\xi,\mu}
)}\subset 
\partial^{\mathrm{u}}\Delta\setminus
\partial^{\mathrm{uu}}\Delta,\quad (\xi,\mu)\in\mathcal{P}.
\end{split}
\end{equation*}
Roughly, 
these relations between the boundaries say that the ``front'' and ``rear cover'' of $ \mathcal{A}_{\xi,\mu}$ and $ \mathcal{B}_{\xi,\mu}$ are contained in the ``front'' and ``rear cover'' of $\Delta$, respectively, (see Figure \ref{Fig:BH}).
%See Figure~\ref{fig:Split}.
}}
\end{remark}
\begin{lemma}
\label{l.L1}
For every $(\xi,\mu)\in \mathcal{P}$ it holds 
%there are two disjoint connected subsets $\mathcal{A}_{\xi,\mu}$ and $\mathcal{B}_{\xi,\mu}$ of $\Delta$ such that
%$G_{\xi,\mu}(\mathcal{A}_{\xi,\mu})$ and $G_{\xi,\mu}(\mathcal{B}_{\xi,\mu})$
% are the connected components of 
\begin{itemize}
\item[a)]
$
G_{\xi,\mu}(\Delta)\cap (\mathbb{R}\times [-4,4]\times \mathbb{R})
=G_{\xi,\mu}(\mathcal{A}_{\xi,\mu})\cup 
G_{\xi,\mu}(\mathcal{B}_{\xi,\mu}),
$ 
\item[b)]
$G_{\xi,\mu}(\mathcal{A}_{\xi,\mu})\cup G_{\xi,\mu}(\mathcal{B}_{\xi,\mu})\subset (-4,4)\times [-4,4]\times\mathbb{R}$.
\end{itemize}
 \end{lemma}

\begin{proof}
We begin showing the equality of the item a).
Keeping in mind Remark~\ref{r.front}, 
the inclusion ``$\subset$''
is obtained from the  relations (see Figure \ref{Fig:BH}):  
 \begin{equation}\label{e.eq1}
 G_{\xi,\mu}(\mathcal{A}_{\xi,\mu}) \cap G_{\xi,\mu}(\mathcal{B}_{\xi,\mu})=\emptyset,\quad
  G_{\xi,\mu}\big(\Delta\setminus
   (\mathcal{A}_{\xi,\mu}\cup
    \mathcal{B}_{\xi,\mu})\big)\cap
     \Delta=\emptyset,\quad (\xi,\mu)\in\mathcal{P}.
 \end{equation}   
The  reciprocal inclusion ``$\supset$'' follows from   the relation: 
\begin{equation}\label{e.eq2}
G_{\xi,\mu}\big(
\partial^{\mathrm{uu}}\mathcal{A}_{\xi,\mu}\cup
  \partial^{\mathrm{uu}}
  \mathcal{B}_{\xi,\mu}\big) 
 \subset
 \{|y|=4\},\quad \quad (\xi,\mu)\in\mathcal{P}.
\end{equation}
% \margem{Q10: formalmente estamos probando solamente el item (a) del lemma..... (b) es por definicion, obvio?...\textcolor{red}{la prueba estaba comentada..de todas formas edit\'e la prueba de \'este lema nuevamente...creo que qued\'o bastante mejor...}}
To get the first relation in \eqref{e.eq1}, it is sufficient to study the projections of $G_{\xi,\mu}(\mathcal{A}_{\xi,\mu})$ and $G_{\xi,\mu}(\mathcal{B}_{\xi,\mu})$ in the plane $\mathbb{XY}$. We denote such projection by $\Pi_3$.
 
\begin{cl}\label{cl.ed} For every $(\xi,\mu)\in\mathcal{P}$ it holds 
 $\Pi_3 (G_{\xi,\mu}(\mathcal{A}_{\xi,\mu}))\cap \Pi_3 (G_{\xi,\mu}(\mathcal{B}_{\xi,\mu}))=\emptyset$.
 \end{cl}
\begin{proof}
Let $(\xi,\mu)\in\mathcal{P}$, then  we have that
\begin{equation*}
\begin{split}
\Pi_3 (G_{\xi,\mu}(\mathcal{A}_{\xi,\mu}))=\big\{(y,\mu+y^2):y \in \I_{\mu}\big\},\,\,
\Pi_3 (G_{\xi,\mu}(\mathcal{B}_{\xi,\mu}))=\big\{(y,\mu+y^2):y \in \J_{\mu}\big\}.
\end{split}
\end{equation*}
From $\I_{\mu}\cap \J_{\mu}=\emptyset$ it follows that $\Pi_3 (G_{\xi,\mu}(\mathcal{A}_{\xi,\mu}))\cap\Pi_3 (G_{\xi,\mu}(\mathcal{B}_{\xi,\mu}))=\emptyset,$ ending the proof of the claim.
\end{proof}

\begin{remark}\label{r.cl}
{\em{
Equation \eqref{e.valoresy} and the proof of the claim above also imply that 
$$\Pi_3 \big( G_{\xi,\mu}(\mathcal{A}_{\xi,\mu})\cup G_{\xi,\mu}(\mathcal{B}_{\xi,\mu}) \big)\subset (-4,4)\times[-4,4],\quad
\mbox{for every  $(\xi,\mu)\in\mathcal{P}.$}
$$ 
}}
\end{remark}

We now prove \eqref{e.eq2} and the second part of \eqref{e.eq1}. Since the endomorphisms $G_{\xi,\mu}$ collapse the $\mathbb{X}$-direction,  
it is sufficient to study the corresponding projections in the plane $\mathbb{YZ}$. 
For this, consider the sets
\begin{equation*}
\begin{split}
&\Gamma_{\mu}\eqdef
 \Big([-4,a_{\mu})\cup
  (b_{\mu},c_{\mu})
  \cup (d_{\mu},4]\Big)\times[-40,22],
\\
&C^1_\mu\eqdef\{a_{\mu}\}
\times[-40,22],\quad 
C^2_\mu\eqdef \{b_{\mu}\}
\times[-40,22],
\\
&C^3_\mu\eqdef\{c_{\mu}\}
\times[-40,22],
\quad
C^4_{\mu}\eqdef \{d_{\mu}\}
\times[-40,22].
\end{split}
\end{equation*}
Note that $\Gamma_{\mu}, C^1_{\mu},C^2_{\mu}$,  $C^3_{\mu}$, and $C^4_{\mu}$ are, respectively, the projections  on the plane $\mathbb{YZ}$ of the sets
\begin{equation*}
\begin{split}
&\Delta\setminus (\mathcal{A}_{\xi,\mu}\cup \mathcal{B}_{\xi,\mu}),
\\
&\partial^{\mathrm{uu}}\mathcal{A}_{\xi,\mu}
\cap \{y=a_{\mu}\},\quad
\partial^{\mathrm{uu}}
\mathcal{A}_{\xi,\mu}
\cap \{y=b_{\mu}\},
\\
&\partial^{\mathrm{uu}}
\mathcal{B}_{\xi,\mu}\cap 
\{y=c_{\mu}\},\quad
\partial^{\mathrm{uu}}\mathcal{B}_{\xi,\mu}
\cap \{y=d_{\mu}\}.
\end{split}
\end{equation*}

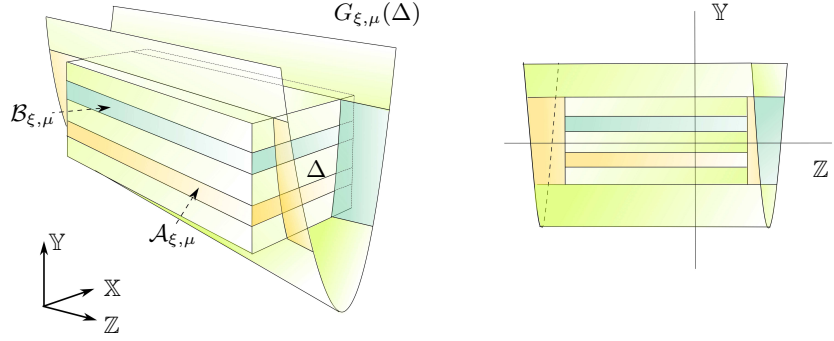
\begin{figure}
\centering
\begin{overpic}[scale=.18,bb=0 0 1928 908,tics=5
  ]{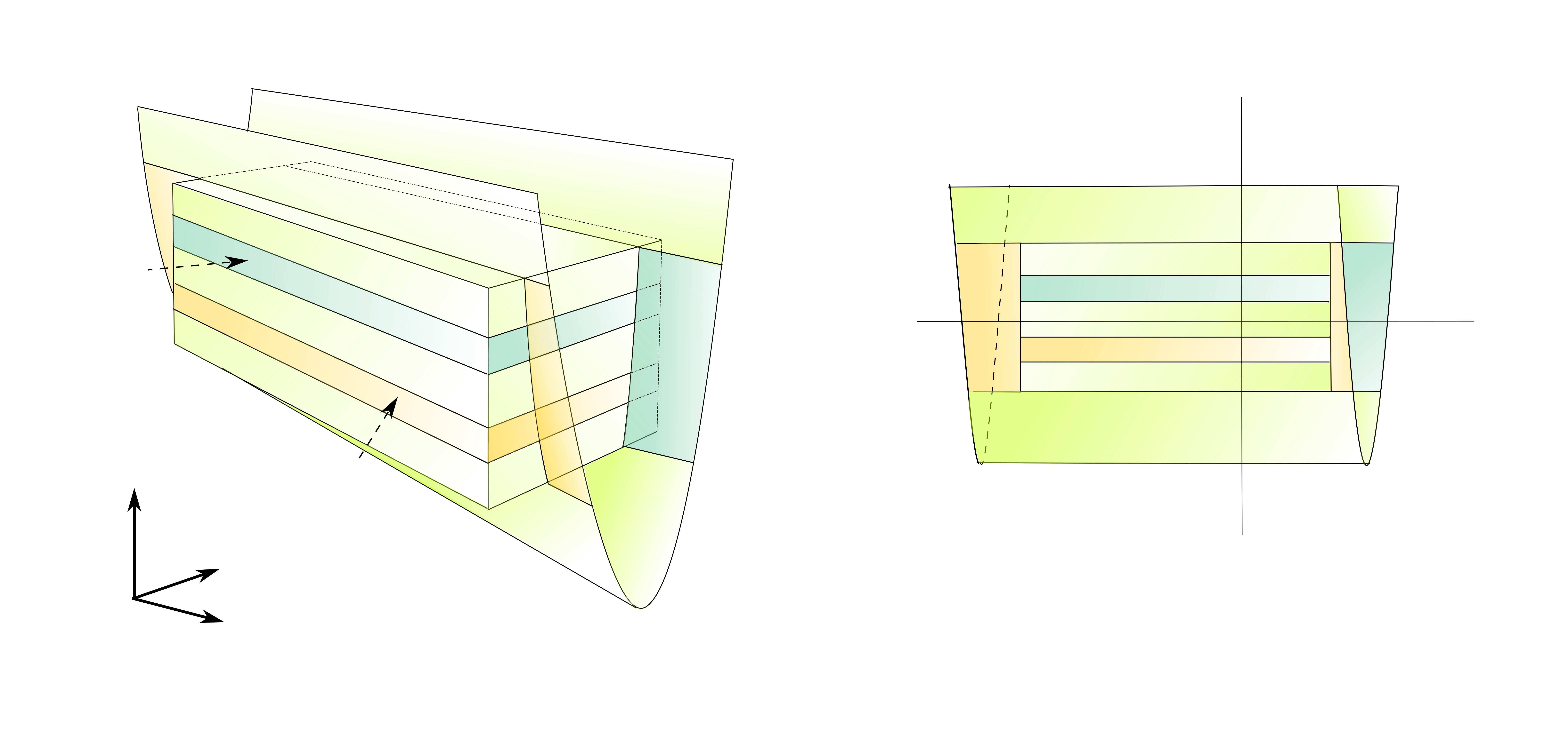}
  \put(40,40){\small{$G_{\xi,\mu}(\Delta)$}} 
\put(37,23){\small{$\Delta$}}
     \put(19,16){\small{ $\mathcal{A}_{\xi,\mu}$}}
      \put(4,29){\small{ $\mathcal{B}_{\xi,\mu}$}}
      \put(9,15){\small{$\mathbb{Y}$}}
        \put(15,10){\small{$\mathbb{X}$}}
        \put(15,6){\small{$\mathbb{Z}$}}
        \put(81,40){\small{$\mathbb{Y}$}}
        \put(92,23){\small{$\mathbb{Z}$}}
         \end{overpic}
\caption{The blender-horseshoe of $G_{\xi,\mu}$ 
}
\label{Fig:BH}
%\vspace{0.2cm}
\end{figure}
 
 Recall the definition of the endomorphism $g_{\xi,\mu}$ in \eqref{e.g}.
 
\begin{cl}For every $(\xi,\mu)\in\mathcal{P}$ it holds that
\begin{itemize}
\item[a')]
$g_{\xi,\mu}(\Gamma_\mu)\cap 
\big([-4,4]\times[-40,22]\big)=\emptyset$,
\item[b')]
$g_{\xi,\mu}(C^1_\mu \cup C^4_\mu)\subset \{y=4\}$, and 
\item[c')]
$g_{\xi,\mu}(C^2_\mu\cup C^3_\mu)\subset \{y=-4\}$. 
\end{itemize}
\end{cl}
\begin{proof}
Consider the projection
$\Pi_{13}(x, y, z) \eqdef y$.
It is easy to check the following equalities:
\begin{equation*}
\begin{split}
\Pi_{13}\Big(g_{\xi,\mu}\big([-4,a_{\mu})
\times[-40,22]\big)\Big)&=(4,\mu+16],
\\
\Pi_{13}\Big(g_{\xi,\mu}
\big((b_{\mu},c_{\mu})\times[-40,22]\big)\Big)&=[\mu,-4),
\\
\Pi_{13}\Big(g_{\xi,\mu}\big((d_{\mu},4]\times[-40,22]\big)\Big)&=(4,\mu+16].
\end{split}
\end{equation*}
Recalling that $\mu \in (-10,-9)$ we get item a'). 
From Remark~\ref{r.Invariant directions} and equation \eqref{e.ys} it follows
\begin{itemize}
\item $g_{\xi,\mu}$ preserves the foliation $\mathcal{F}=\{\{y\}\times\mathbb{R}:y\in \mathbb{R}\}$, and
\item 
$
\mu+a_{\mu}^2=
\mu+d_{\mu}^2=-(\mu+b_{\mu}^2)
=-(\mu+c_{\mu}^2)
=4. 
$
\end{itemize}
These two facts imply items b') and c').
This ends the proof of the claim.
 \end{proof}
 
 The proof of item a) of the lemma is now complete.
 Finally, item b) follows directly from Remark \ref{r.cl}. 
%From here it follows that 
%$$
%\Pi_3\big(G_{\xi,\mu}(\Delta)\big)\cap [-4,4]^2=\Pi_3\big(G_{\xi,\mu}(\mathcal{A}_{\xi,\mu})\cup G_{\xi,\mu}(\mathcal{B}_{\xi,\mu})\big).
%$$ 
%
%
%and that $y\in [-4,4]\setminus(I_{\mu}\cup J_{\mu})$ imply $|\mu+y^2|>4$. More precisely, it holds that
%\begin{equation*}
%\begin{split}
%y\in [-4,a_{\mu})\cup (d_{\mu},4] \Rightarrow \mu+y^2>4
%\quad\mbox{and}\quad 
%y\in (b_{\mu},c_{\mu}) \Rightarrow \mu+y^2<-4.
%\end{split}
%\end{equation*}
%
%
%Observe that
%\begin{equation}
%\label{e.curvas}
%\begin{split}
%\Pi_3\big(G_{\xi,\mu}(\mathcal{A}_{\xi,\mu})\big)&:=\big\{(y,\mu+y^2):y\in I_{\mu}\big\},
%\\
%\Pi_3\big(G_{\xi,\mu}(\mathcal{A}_{\xi,\mu})\big)&:=\big\{(y,\mu+y^2):y\in J_{\mu}\big\}.
%\end{split}
%\end{equation} 
%such that $\ell_{+,\mu}\cup \ell_{-,\mu}\subset (-4,4)\times[-4,4]$.
%
%Therefore, for every $(\xi,\mu)\in\mathcal{P}$ it holds
%$$
%G_{\xi,\mu}(\Delta)\cap [-4,4]^2=
%(\ell_{+,\mu}\cup \ell_{-,\mu})\times[-40,22].
%$$
%$$
%G_{\xi,\mu}(\mathcal{A}_{\mu})\cup G_{\xi,\mu}(\mathcal{B}_{\mu})\subset (-4,4)\times [-4,4]\times\mathbb{R},
%\quad(\mathcal{A}_{\mu}\cup \mathcal{B}_{\mu})\cap \partial^{\mathrm{uu}} \Delta=\emptyset.
%$$ 
The proof of the lemma is now complete.
 \end{proof}

\subsection{Contracting/expanding  invariant cone fields}In this section, we study the condition (\BHb) of a blender-horseshoe
involving invariance, contraction, and expansion
of the cone fields in \eqref{e.conefilds}. This condition is a consequence of the following lemma.

\begin{lemma}
\label{l.condition3}
Let $\vartheta>0$ and $\theta=1/2$. Then,
for every $(\xi,\mu)\in \mathcal{P}$
 and every $p\in\mathcal{A}_{\xi,\mu}\cup\mathcal{B}_{\xi,\mu} $ the following holds:
\begin{itemize}
\item[(i)]
 $\mathcal{C}^{\mathrm s}_{\vartheta}\big(G_{\xi,\mu}(p)\big)\subset D(G_{\xi,\mu})_p\big(\mathcal{C}^\mathrm s_{\vartheta}(p)\big)$,
\item[(ii)]
 $D(G_{\xi,\mu})_p\big(\mathcal{C}^\mathrm{u}_{\theta}(p)\big)\subset \mathcal{C}^\mathrm{u}_\theta\big(G_{\xi,\mu}(p)\big)$,
\item[(iii)]
$D(G_{\xi,\mu})_p\big(\mathcal{C}^\mathrm{uu}_\theta(p)\big)\subset \mathcal{C}^\mathrm{uu}_\theta\big(G_{\xi,\mu}(p)\big)$,
\item[(iv)]
 $DF|_{\mathcal{C}^{\mathrm u}_\theta}$  is uniformly expanding and $DF|_{\mathcal{C}^{\mathrm s}_{\vartheta}}$ is  
uniformly contracting for every 
$\vartheta$ sufficiently small. 
\end{itemize}
\end{lemma}
%
%\begin{remark}
%{\em{
%Note that if $F$ is sufficiently $C^1$-close to $G_{\mu.\xi}$ and $A_F$
%and $B_F$ are the continuations of
%$A'_{\mu,\xi}$ and $B'_{\mu,\xi}$
%respectively we have that 
%$$
%\mbox{
%$p\in F({A_F})\cup F({B}_F)$
%\quad then\quad $DF^{-1}(\mathcal{C}^{\mathrm s}_2(p))\subset \mathcal{C}^\mathrm s_2(F^{-1}(p))$.
%}
%$$
%}}
%\end{remark}

\begin{proof}
Consider  $p=(x,y,z)\in \mathcal{A}_{\xi,\mu}\cup\mathcal{B}_{\xi,\mu} $ and $\textbf{v} =
(u, v, w)\in T_{p}\Delta$, write
$$
(u_1, v_1, w_1)\eqdef D(G_{\xi,\mu})_p(\textbf{v}) = (v, 2\,y\,v, v + \xi\,w).
$$
Recalling ~\eqref{e.AmBm} and \eqref{e.estim}, we have that if $(x,y,z)\in \mathcal{A}_{\xi,\mu}\cup\mathcal{B}_{\xi,\mu}$ then 
$y\in \I_{\mu}\cup\J_{\mu}$ and thus  
$|y|>\sqrt{5}$, for every $\mu\in(-10,-9)$.

The  items  of the lemma are proved in the following claims.

\begin{cl}[Item (i)]
\label{cl.(i)} Let $\vartheta>0$. For every $\mathbf{v}\in\partial \mathcal{C}^{\mathrm s}_{\vartheta}(p)\setminus \{ \bar 0\}$ we have
$ D(G_{\xi,\mu})_p(\mathbf{v})\in
\big(\overline{\mathcal{C}^{\mathrm s}_{\vartheta}(G_{\xi,\mu}(p)\big)}\big)^c.
$
\end{cl}
\begin{proof}
If $\mathbf{v}\in\partial \mathcal{C}^{\mathrm s}_{\vartheta} (p)\setminus \{ \bar 0\}$ then $\vartheta\, (\sqrt{v^2+w^2})=|u|$. Since
$|y|>\sqrt{5}$, we get that
$$
\vartheta\, \big( \sqrt{v_1^2+w_1^2} \big)\ge \vartheta \, |v_1|>2 \, |y|\, |v|>
2\,\sqrt{5}\,|v|= 2\,\sqrt{5}\,|u_1|>
|u_1|.
$$
Therefore $ D(G_{\xi,\mu})_p(\mathbf{v})\notin
\overline{\mathcal{C}^{\mathrm s}_{\vartheta}(G_{\xi,\mu}(p)\big)}
$,
proving the claim.
\end{proof}

\begin{cl}[Item (ii)]
\label{cl.(ii)}
For every 
$\mathbf{v}\in \mathcal{C}_{1/2}^{\mathrm {u}}(p)$ it holds 
$D(G_{\xi,\mu})_p ({\mathbf{v}})\in \mathcal{C}_{1/2}^{\mathrm {u}}\big(G_{\xi,\mu}(p)\big)$.
\end{cl}
\begin{proof}
Since 
$|y|>\sqrt{5}$, we have that
$$
\sqrt{v_1^2+w_1^2}\ge |v_1|=2\, |y|\, |v|>
2\, \sqrt{5}\, |v|>
2\, |u_1|,
$$
proving the claim.
\end{proof}

\begin{cl}[Item (iii)]
For every 
$\mathbf{v}\in \mathcal{C}_{1/2}^{\mathrm {uu}}(p)$ it holds 
$D(G_{\xi,\mu})_p ({\mathbf{v}})\in \mathcal{C}_{1/2}^{\mathrm {uu}}\big(G_{\xi,\mu}(p)\big)$.
\end{cl}
\begin{proof}
We need to check that 
%if
%$\textbf{v}=(u,v,w)$ and 
%$D(G_{\xi,\mu})_p(\textbf{v})=(u_1,v_1,w_1)=(v, 2yv, v + \xi\,w)$ it holds that
$$
\sqrt{u^2+w^2}<\frac1{2}\,|v|\quad\Rightarrow\quad
\sqrt{u_1^2+w_1^2}<\frac1{2}\, |v_1|.
$$
Note that $\sqrt{u^2+w^2}<\frac1{2}\, |v|$ implies that $|w|<\frac1{2}\,|v|$, and hence
$$
u_1^2+w_1^2
=v^2+(v+\xi\,w)^2\le 
2\,v^2+ 2\,\xi\,|v|\, |w|+\xi^2\,|w|^2
\le
\left(2+\xi+\left(\frac{\xi}{2}\right)^2\right)\,v^2.
$$
Now $\xi\in(1.18,1.19)$ implies that
$$
\left(2+\frac{\xi}{2}+\left(\frac{\xi}{2}\right)^2\right)< 4
$$
and hence
$$
u_1^2+w_1^2< 4 v^2.
$$
Thus, since 
 $p=(x,y,z)\in \mathcal{A}_{\xi,\mu}\cup\mathcal{B}_{\xi,\mu}$ implies that $|y|>\sqrt{5}$, it follows
 $$
2\,\sqrt{u_1^2+w_1^2}<4\, |v|<2\, |y|\,|v|=|v_1|,
$$
proving the claim.
\end{proof}
\begin{cl}[Item (iv)]
 $DG_{\xi,\mu}|_{\mathcal{C}_{1/2}^{\mathrm u}}$ is uniformly expanding and, if
$\vartheta$ is small enough, 
 $DG_{\xi,\mu}|_{\mathcal{C}_\vartheta^{\mathrm s}}$ is 
 uniformly contracting.
\end{cl}
\begin{proof} 
The uniform contraction of the
cone field
$\mathcal{C}_\vartheta^{\mathrm s}$ for small $\vartheta$ 
 follows from the fact that $D(G_{\xi,\mu})_p$ 
  is an endomorphism 
  whose eigenspace
associated the eigenvalue $0$ is spanned by $(1, 0, 0)$.

To see that $D(G_{\xi,\mu})$ uniformly expands the vectors in ${\mathcal{C}_{1/2}^{\mathrm u}}$
consider the norm
$$
|(u, v, w)|_* \eqdef \max\left\{ |u|,
\sqrt{
v^2 + w^2}
\right\}.
$$ 
Take $\textbf{v}=(u, v, w)\in\mathcal{C}^{\mathrm u}_{1/2}(p)$ and write
$D(G_{\xi,\mu})_p(\textbf{v})=(u_1, v_1, w_1)=(v,2\,y\,v, v+\xi \, w)$.
We claim that 
if 
$\textbf{v}\in \mathcal{C}^{\mathrm u}_{1/2}(p)$ then
$|(D(G_{\xi,\mu})_p\textbf{v}|_* > |\textbf{v}|_*$. 
By compactness,
this implies that 
$|(D(G_{\xi,\mu})_p\textbf{v}|_* > c_0\, |\textbf{v}|_*$,
for some uniform $c_0> 1$.  Note that the Euclidean norm $||\cdot||$ and  $|\cdot|_\ast$ are equivalent, hence
there is $\kappa>1$ such that $\kappa^{-1} || \mathbf{v}||  \le |\mathbf{v}|_\ast \le \kappa || \mathbf{v}||$.
The number $\ell$ in (\BHb) is the first $\ell_0$ with $c_0^{\ell_0}> \kappa$. 

%Since $|\cdot|$ and $|\cdot|_*$ we get that 
%there are $c > 1$ such that 
%$|(D(G_{\mu,\xi})_p\textbf{v}|> c|\textbf{v}|$.
% We now go to the details of this.

We now prove that
$|(D(G_{\xi,\mu})_p\textbf{v}|_* > |\textbf{v}|_*$. 
 Note that for $\mathbf{v}=(u,v,w)\in \mathcal{C}^{\mathrm u}_{1/2}(p)$ we have $|\textbf{v}|_*=\sqrt{v^2+w^2}$ and 
\begin{equation}\label{e.cris}
v_1^2+w_1^2
= 4\, v^2\, y^2 +
(v + \xi\,w)^2\ge
4\,v^2\,y^2
+
v^2 -
 2\,\xi\,|v|\,|w| +
\xi^2\, w^2.
\end{equation}
We divide the proof into two cases: $(6.5)\, |v| \ge  |w|$ and $(6.5)\,|v| \le  |w|$.
If
$(6.5)\,|v| \ge  |w|$, using that  $\xi\in (1.18,1.19)$ and
$|y|>\sqrt{5}$, we get that
\begin{equation}\label{e.crisb}
4\, v^2\, y^2
-
 2\, \xi\,|v|\,|w|\ge 
 (20
-
 13\, \xi)\,v^2>4\,v^2\ge 0.
\end{equation}
Equations~\eqref{e.cris} and \eqref{e.crisb} immediately imply that
$$
v_1^2+w_1^2> 5\, v^2+\xi^2\, w^2> v^2+ w^2.
$$
Hence, $|(D(G_{\xi,\mu})_p\textbf{v}|_* > |\textbf{v}|_*$, proving the first case.
Similarly, if $(6.5)\, |v| \le |w|$ then
$$
v_1^2+w_1^2 \ge
4\,y^2\,v^2+ 
\xi^2\,w^2-2\,\xi\,|v|\,|w|+ v^2  
>
4y^2v^2+ 
\xi^2\,w^2-2\,\xi\,(6.5)^{-1}\,w^2+ v^2.
$$
Condition $\xi\in (1.18,1.19)$ implies that
$$
\xi^2-2\,\xi\,(6.5)^{-1}>1.
$$
Thus
$$
v_1^2+w_1^2 \ge v^2+w^2.
$$
Hence, $|(D(G_{\xi,\mu})_p\textbf{v}|_* > |\textbf{v}|_*$.
%We have obtained that $|(D(G_{\xi,\mu})_p\textbf{v}|_* \ge c_0 |\textbf{v}|_*$ for some $c_0>1$.
 This ends the proof of the claim.
\end{proof}
The proof of the lemma is now complete.
\end{proof}

\begin{remark}\label{r.reversing}
{\em{
For each $p=(x,y,z)\in \mathbb{R}^3$ we identify $T_p\mathbb{R}^3$ with  $\mathbb{R}^3$ and 
consider the canonical basis $\{\textbf{i}, \textbf{j}, \textbf{k}\}$.
Note  that $D(G_{\xi,\mu})_p(\textbf{i})=\textbf{0}$, $D(G_{\xi,\mu})_p(\textbf{j})=\textbf{i}+2\,y\,\textbf{j}_2+\textbf{k}$, and  
$D(G_{\xi,\mu})_p(\textbf{k})=\xi\,\textbf{k}$. In particular, 
 $\langle D(G_{\xi,\mu})_p(\textbf{j}), \textbf{j}\rangle<0$ (resp. $>0$) if $y<0$ 
 (resp. $y>0$).
As a consequence, for every $\theta>0$ and every $p\in\mathcal{A}_{\xi,\mu}$, the derivative  $D(G_{\xi,\mu})_p$ sends the semi-positive  cone $\mathcal{C}^{\mathrm{uu}}_\theta(p)\cap\{y>0\}$ (resp. semi-negative cone)
into the semi-space $\{y<0\}$ (resp.  $\{y>0\}$). When $p\in\mathcal{B}_{\xi,\mu}$
the derivative  $D(G_{\xi,\mu})_p$ maps the semi-positive cone $\mathcal{C}^{\mathrm{uu}}_\theta(p)\cap\{y>0\}$ (resp. semi-negative cone) into $\{y>0\}$ (resp. $\{y>0\}$).
}}
\end{remark}

\subsection{The Markov partition}\label{ss.mp}
To define the Markov partition in Condition (\BHc) we need some preliminary
 constructions. 

For $(\xi,\mu)\in\mathcal{P}$ consider the auxiliary straight lines $R^1_{\xi,\mu},R^2_{\xi,\mu}$ in the plane $\mathbb{YZ}$ defined by the equations and depicted in Figure~\ref{fig:Markov},
\begin{equation*}
\begin{split}
&R^1_{\xi,\mu} \eqdef \big\{ \big(y,z^1_{\xi}(y)\big): z^1_{\xi}(y)=\xi^{-1}\,(22-y), \, y\in \mathbb{R}\big\},\\
&R^2_{\xi,\mu} \eqdef \big\{
\big(y,z^2_{\xi}(y)\big)
: z^2_{\xi}(y)=\xi^{-1}\,(-40-y),  \, y\in \mathbb{R}\big\}.
\end{split}
\end{equation*} 
 Recall the definition of the intervals $\I_{\mu}=[a_{\mu},b_{\mu}]$ and $\J_{\mu}=[c_{\mu},d_{\mu}]$ in \eqref{e.valoresy}.  
 Consider the auxiliary parallelogram $\mathtt{A}_{\xi,\mu}$ in the plane  $\mathbb{Y}\mathbb{Z}$
whose boundary consists of the following  segments (see Figure~\ref{fig:Markov}):
\[
\begin{split}
&\L^1_{\xi,\mu}\eqdef \big\{\big(y,z^1_{\xi}(y)\big):
 y\in \I_{\mu}\big\},\quad 
 \L^2_{\xi,\mu}\eqdef \{a_{\mu}\}
\times\big[z^2_{\xi}(a_{\mu}),z^1_{\xi}(a_{\mu})\big],\\
& 
\L^3_{\xi,\mu}\eqdef \big\{\big(y,z^2_{\xi}(y)\big):
 y\in \I_{\mu} \big\},
\quad \L^4_{\xi,\mu}\eqdef \{b_{\mu}\}
\times\big[z^2_{\xi}(b_{\mu}),z^1_{\xi}(b_{\mu})\big].
\end{split}
\]
%\textcolor{red}{
%\[
%\begin{split}
%&\L^1_{\xi,\mu}\eqdef \big\{\big(y,\xi^{-1}\,(22-y)\big):
% y\in \I_{\mu}\big\},\quad 
% \L^2_{\xi,\mu}\eqdef \{a_{\mu}\}
%\times\big[\xi^{-1}\,(-40-a_{\mu}),\xi^{-1}\,
%(22-a_{\mu})\big],\\
%& 
%\L^3_{\xi,\mu}\eqdef \big\{\big(y,\xi^{-1}
%\,(-40-y)\big):
% y\in \I_{\mu} \big\},
%\quad \L^4_{\xi,\mu}\eqdef \{b_{\mu}\}
%\times\big[\xi^{-1}\,(-40-b_{\mu}),
%\xi^{-1}\,(22-b_{\mu})\big].
%\end{split}
%\]
%}
Analogously, consider
the parallelogram  $\mathtt{B}_{\xi,\mu}$
in the plane $\mathbb{Y}\mathbb{Z}$
bounded by
\[
\begin{split}
&\tilde \L^1_{\xi,\mu}\eqdef \big\{\big(y,z^1_{\xi}(y)\big):
 y\in \J_{\mu}\big\},\quad 
 \tilde \L^2_{\xi,\mu}\eqdef \{c_{\mu}\}
\times\big[z^2_{\xi}(c_{\mu}),z^1_{\xi}(c_{\mu})\big],\\
& 
\tilde \L^3_{\xi,\mu}\eqdef \big\{\big(y,z^2_{\xi}(y)\big):
 y\in \J_{\mu} \big\},
\quad \tilde \L^4_{\xi,\mu}\eqdef \{d_{\mu}\}
\times\big[z^2_{\xi}(d_{\mu}),z^1_{\xi}(d_{\mu})\big].
\end{split}
\]
\begin{figure}
\centering
\begin{overpic}[scale=.25,bb=0 1 1389 653,tics=5
  ]{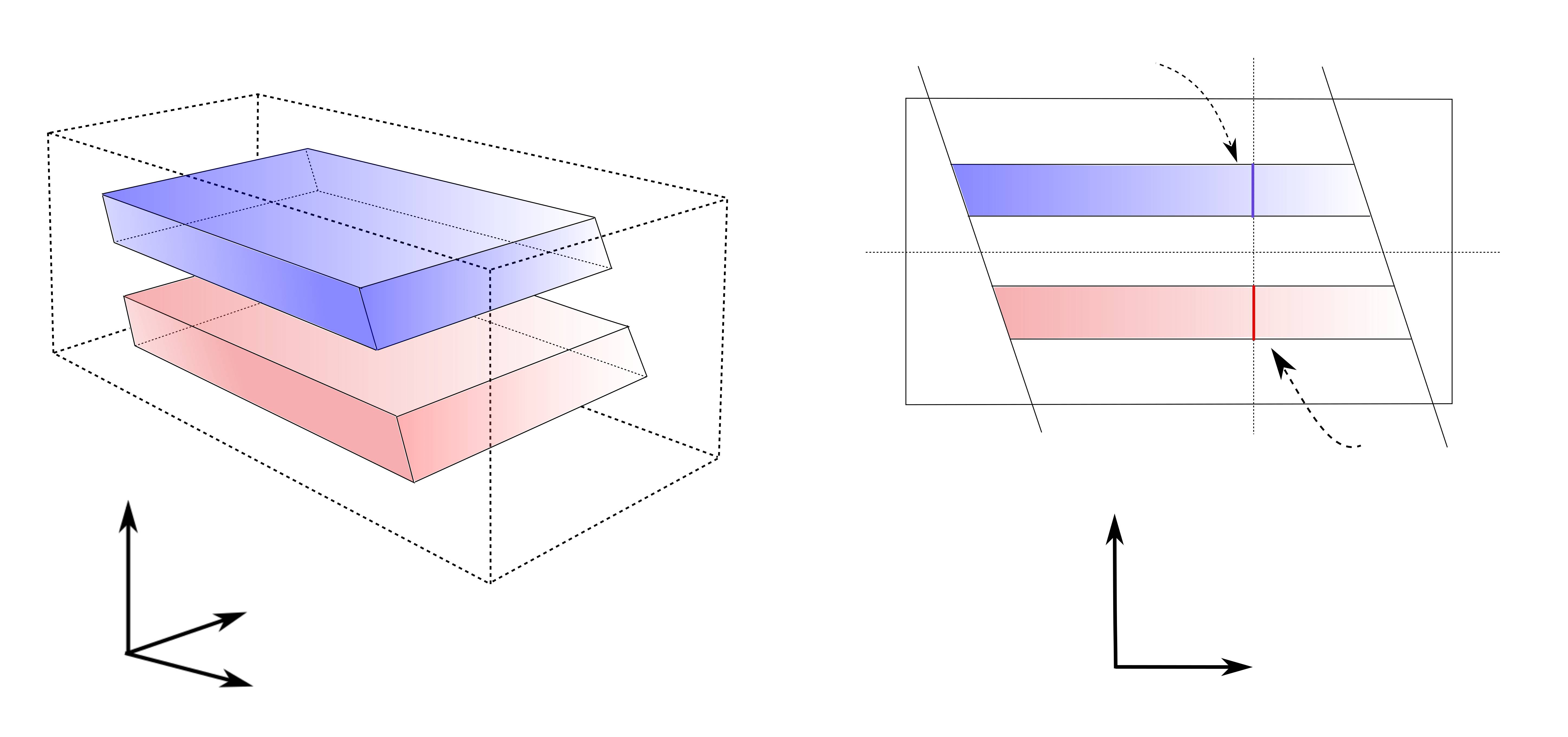}
     \put(41,22){{ $\mathbb{A}_{\xi,\mu}$}}
       \put(39,30){{ $\mathbb{B}_{\xi,\mu}$}}
     \put(68,26.5){{$\mathtt{A}_{\xi,\mu}$}}
       \put(65,34.3){{$\mathtt{B}_{\xi,\mu}$}}
                 \put(17,1.7){\small{$ \mathbb{Z}$}}
                  \put(9,13){\small{$ \mathbb{Y}$}}
                     \put(16.5,7.4){\small{$ \mathbb{X}$}}
                  \put(85,43.4){\small{$R^1_{\xi,\mu}$}}
                     \put(59,43.5){\small{$R^2_{\xi,\mu}$}}
                        \put(87,17.5){\small{$\I_{\mu}$}}
                          \put(71,42.5){\small{$\J_{\mu}$}}
                                 \put(77.5,5.8){\small{$ \mathbb{Z}$}}
                                      \put(72,12){\small{$ \mathbb{Y}$}}
 \end{overpic}
\caption{The Markov partition of the blender-horseshoe.}
\label{fig:Markov}
%\vspace{0.2cm}
\end{figure}
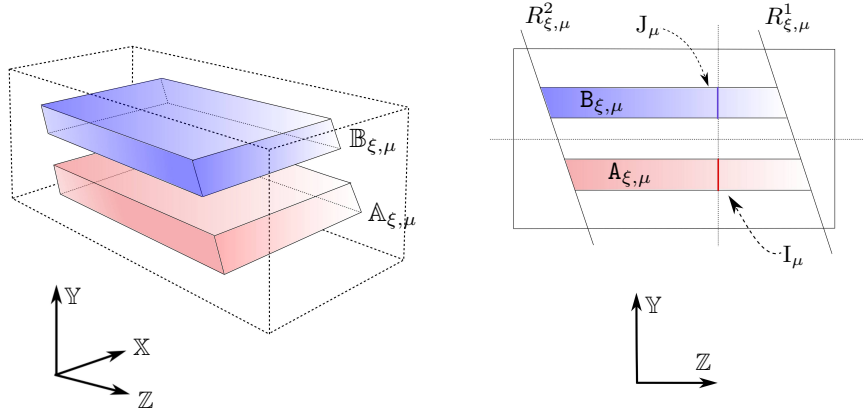
\begin{remark}
\label{r.erre}
{\em{
Since $(\xi,\mu) \in \mathcal{P}$, it follows  that $\mathtt{A}_{\xi,\mu}$ and $\mathtt{B}_{\xi,\mu}$ are contained in
$ (-4,0)\times (-40,22)$.
By the definitions of $\mathtt{A}_{\xi,\mu}$ and $\mathtt{B}_{\xi,\mu}$, it holds that 
$$g_{\xi,\mu} (\partial \mathtt{A}_{\xi,\mu})=
g_{\xi,\mu} \big(\cup^4_{i=1} \L^i_{\xi,\mu}\big)
=
\partial \big([-4,4]\times[-40,22]\big),$$
$$
g_{\xi,\mu} (\partial \mathtt{B}_{\xi,\mu})=
g_{\xi,\mu} \big(\cup^4_{i=1} \tilde{\L}^i_{\xi,\mu}\big)
=
\partial \big([-4,4]\times[-40,22]\big),
$$
and thus 
$$
g_{\xi,\mu}(\mathtt{A}_{\xi,\mu})=g_{\xi,\mu}(\mathtt{B}_{\xi,\mu})=[-4,4]\times[-40,22].
$$
}}
\end{remark}

We now show that the sets $\mathbb{A}_{\xi,\mu}$ and $\mathbb{B}_{\xi,\mu}$ (see Figure \ref{fig:Markov})
\begin{equation*}
%\label{e.finalsets}
\begin{split}
\mathbb{A}_{\xi,\mu}\eqdef [-4,4]\times\mathtt{A}_{\xi,\mu}\quad
\mbox{and}\quad
\mathbb{B}_{\xi,\mu}\eqdef [-4,4]\times\mathtt{B}_{\xi,\mu}.
\end{split}
\end{equation*}
 form a Markov partition of the blender-horseshoe of $G_{\xi,\mu}$ in $\Delta$. 
 Observe first that 
\begin{equation*}
\mathbb{A}_{\xi,\mu}=(G_{\xi,\mu}|_{\Delta})^{-1}\big(G_{\xi,\mu}(\mathcal{A}_{\xi,\mu})\cap\Delta)\big),\quad
\mathbb{B}_{\xi,\mu}=(G_{\xi,\mu}|_{\Delta})^{-1}\big(G_{\xi,\mu}(\mathcal{B}_{\xi,\mu})\cap\Delta)\big).
\end{equation*}
The next lemma completes the proof of condition  (\BHc).
%As in previous cases,
%we split the boundary of $\mathbb{A}_{\xi,\mu}$ (similarly for $\mathbb{B}_{\xi,\mu}$) 
%$$
%\partial \mathbb{A}_{\xi,\mu}\eqdef
%\partial^{\mathrm{u}} \mathbb{A}_{\xi,\mu}\cup \partial^{\mathrm{s}} \mathbb{A}_{\xi,\mu},
%$$
%and consider the subset $\partial^{\mathrm{uu}} \mathbb{A}_{\xi,\mu}$ of $\partial^{\mathrm{u}} \mathbb{A}_{\xi,\mu}$,
%where
%\margem{Q12: $\partial^{\mathrm{uu}} \mathbb{A}_{\xi,\mu}=
%[-4,4]\times (\L^2_{\xi,\mu}\cup \L^4_{\xi,\mu})$ estaba equivocado...\textcolor{red}{esta correcto....esta frontera consiste de las tapas superior e inferior de $\mathbb{A}_{\xi,\mu}$} }
%\begin{equation}
%\begin{split}
%\partial^{\mathrm{u}} \mathbb{A}_{\xi,\mu}&\eqdef
%[-4,4]
%\times (
%\L^1_{\xi,\mu}\cup \L^2_{\xi,\mu}\cup
%\L^3_{\xi,\mu}\cup \L^4_{\xi,\mu}
%),\\
%\partial^{\mathrm{s}} \mathbb{A}_{\xi,\mu}&\eqdef \partial([-4,4])\times \mathtt{A}_{\xi,\mu},
%\\
%\partial^{\mathrm{uu}} \mathbb{A}_{\xi,\mu}&\eqdef[-4,4]\times (\L^2_{\xi,\mu}\cup \L^4_{\xi,\mu}).
%\end{split}
%\end{equation}

\begin{lemma}
\label{l.L2}
 For
every $(\xi,\mu)\in\mathcal{P}$ the following holds
%$$\mbox{
%$\mathbb{A}_{\xi,\mu}\subset 
%G^{-1}_{\xi,\mu}\big(G_{\xi,\mu}(\mathcal{A}_{\xi,\mu})\cap\Delta\big)
%$\quad and\quad $\mathbb{B}_{\xi,\mu}\subset 
%G^{-1}_{\xi,\mu}\big(G_{\xi,\mu}(\mathcal{B}_{\xi,\mu})\cap\Delta\big),
%$}$$ 
\begin{itemize}
\item[a)]
$\mathbb{A}_{\xi,\mu}\cup\mathbb{B}_{\xi,\mu}\subset [-4,4]\times(-4,4)\times(-40,22)$,
\item[b)]
$G_{\xi,\mu}(\mathbb{A}_{\xi,\mu})\cup G_{\xi,\mu}(\mathbb{B}_{\xi,\mu})\subset (-4,4)\times[-4,4]\times \mathbb{R}$.
\end{itemize}

\end{lemma}

\begin{proof}
Item a) follows from Remark~\ref{r.erre}. For item b),  note that Lemma~\ref{l.L1} implies that
\begin{equation*}
G_{\xi,\mu}(\mathbb{A}_{\xi,\mu})\cup G_{\xi,\mu}(\mathbb{B}_{\xi,\mu})
\subset 
G_{\xi,\mu}(\mathcal{A}_{\xi,\mu})\cup G_{\xi,\mu}(\mathcal{B}_{\xi,\mu})
\subset (-4,4)\times[-4,4]\times \mathbb{R},
\end{equation*}
 completing of proof of lemma.
%$G_{\xi,\mu}(\mathbb{A}_{\xi,\mu})
%=G_{\xi,\mu}(\mathcal{A}_{\xi,\mu})\cap\Delta$
%and that
%$G_{\xi,\mu}(\mathbb{B}_{\xi,\mu})
%=G_{\xi,\mu}(\mathcal{B}_{\xi,\mu})\cap\Delta$. On the other hand, we have that
%$$G_{\xi,\mu}(\mathcal{A}_{\xi,\mu}\cup 
%\mathcal{B}_{\xi,\mu}
%)\cap\Delta= (\ell_{-,\mu}\cup \ell_{+,\mu})\times[-40,22],$$
%where
%\begin{equation*}
%\begin{split}
%\ell_{-,\mu}:=
%\big\{(y,\mu+y^2):y\in 
%[a_{\mu},b_{\mu}]\big\},\quad
%\ell_{+,\mu}:=
%\big\{(y,\mu+y^2):y\in 
%[c_{\mu},d_{\mu}]\big\}.
%\end{split}
%\end{equation*}
%Since the projection at
%$\mathbb{X}$-coordinate of the sets 
%$\ell_{-,\mu}$ and $\ell_{+,\mu}$ are, 
%respectively, the intervals  
%$[a_{\mu},b_{\mu}]$ and 
%$[c_{\mu},d_{\mu}]$, follows from equality\eqref{e.valoresy} that 
%$[a_{\mu},b_{\mu}]\cup[c_{\mu},d_{\mu}]
%\subset (-4,4)$
%and thus 
%$G_{\xi,\mu}(\mathbb{A}_{\xi,\mu})\cup G_{\xi,\mu}(\mathbb{B}_{\xi,\mu})\subset (-4,4)\times[-4,4]\times \mathbb{R}$, completing of proof of lemma.
\end{proof}
%\begin{remark}
%\label{r.Rt}
%{\em{
%Note that 
%$\tilde p_{\xi,\mu}<\xi^{-1}\,(22-d_{\mu})$ and
%$\xi^{-1}\,(-40-a_{\mu})<\tilde{q}_{\xi,\mu}$. 
%These conditions imply that the
% projection into the $\mathbb{Z}$-coordinate of
%$
%\L^4_{\xi,\mu}\cup \tilde{\L}^3_{\xi,\mu} 
%$
%covers the interval
%$[\tilde q_{\xi,\mu},\tilde p_{\xi,\mu}]$.
% See Figure~\ref{fig:MP}.
%}}
%\end{remark}
\subsection{uu-discs through the local stable manifolds}
\label{ss.through} 
We study Condition (\BHd) of blender horseshoes about the relative position of the $\mathrm{uu}$-discs through the local stable manifolds of 
$P_{\xi,\mu}=(p_{\mu}, p_{\mu},\tilde{p}_{\xi,\mu})$ and $Q_{\mu}=(q_{\mu}, q_{\mu},\tilde{q}_{\xi,\mu})$
with respect to the boundary of $\Delta$. We reduce this analysis to the two dimensional case by projecting these discs on the plane $\mathbb{YZ}$.
Consider the projection
\begin{equation*}
\Pi_1 :\mathbb{R}^3\to \mathbb{R}^2, \quad \Pi_1(x,y,z)\eqdef(y,z).
\end{equation*}
Recalling the formulae for the stable manifolds $W^{\mathrm s}(P_{\xi,\mu})$ and $W^{\mathrm s}(Q_{\xi,\mu})$ in \eqref{e.W}, we get
 $\Pi_1(W^{\mathrm s}(P_{\xi,\mu}))=(p_{\mu},\tilde{p}_{\xi,\mu})$ and 
 $\Pi_1(W^{\mathrm s}(Q_{\xi,\mu}))=(q_{\mu},\tilde{q}_{\xi,\mu})$.

Consider the auxiliary straight lines in the plane $\mathbb{YZ}$ 
through $(p_{\mu},\tilde{p}_{\xi,\mu})$ and $(q_{\mu},\tilde{q}_{\xi,\mu})$:
\begin{equation*}
\begin{split}
L^1_{\xi,\mu} &\eqdef\Big\{\big(y, z^1_{\xi,\mu}(y)\big) \colon z^1_{\xi,\mu}(y)=\frac1{2}(y-p_{\mu})+\tilde{p}_{\xi,\mu}, \, y\in \mathbb{R}\Big\},\\
L^2_{\xi,\mu} &\eqdef \Big\{ \big(y, z^2_{\xi,\mu}(y)\big) \colon z^2_{\xi,\mu}(y)=\frac1{2}(y-q_{\mu})+\tilde{q}_{\xi,\mu},  \, y\in \mathbb{R}\Big\}.
\end{split}
\end{equation*}
Note that $L^1_{\xi,\mu}$ and $L^2_{\xi,\mu}$ are  contained in the boundary of $\Pi_1\big(C^{\mathrm{uu}}_{1/2}(P_{\xi,\mu})\big)$
% where  $C^{\mathrm{uu}}_{1/2}(P_{\xi,\mu})$ it is the uu-cone of size $1/2$ at $P_{\xi,\mu}$ (see \eqref{e.conefilds}).
and of $\Pi_1\big(C^{\mathrm{uu}}_{1/2}(Q_{\xi,\mu})\big)$, respectively. These conditions are depicted in
Figure \ref{fig:6}. 
Thus  (\BHd)
follows now from the next lemma.
\begin{lemma}
For every $(\xi,\mu)\in\mathcal{P}$ it holds that
\begin{equation*}
%\begin{split}
%\label{e.cla}
L^1_{\xi,\mu}\cap \big(\Pi_1(\Delta)\cap \{z=22\}\big)=\emptyset,
\quad
L^2_{\xi,\mu}\cap \big(\Pi_1(\Delta)\cap \{z=-40\}\big)=\emptyset.
%\end{split}
\end{equation*}
\end{lemma}

\begin{proof}
To prove the lemma it is enough to check that %
\begin{equation*}
%\label{e.zz}
z^1_{\xi,\mu}(4)<22\quad\mbox{and}\quad z^2_{\xi,\mu}(-4)>-40,\quad\mbox{ for every $(\xi,\mu)\in\mathcal{P}$}.
\end{equation*}
The choice of parameters $(\xi,\mu)$ and the estimates of $p_{\mu}, q_{\mu}, \tilde{p}_{\xi,\mu}, \tilde{q}_{\xi,\mu}$ in \eqref{e.est}, 
lead directly to these inequalities. 
\end{proof}
\begin{figure}
\centering
\begin{overpic}[scale=.35,bb=0 1 624 456,tics=5
  ]{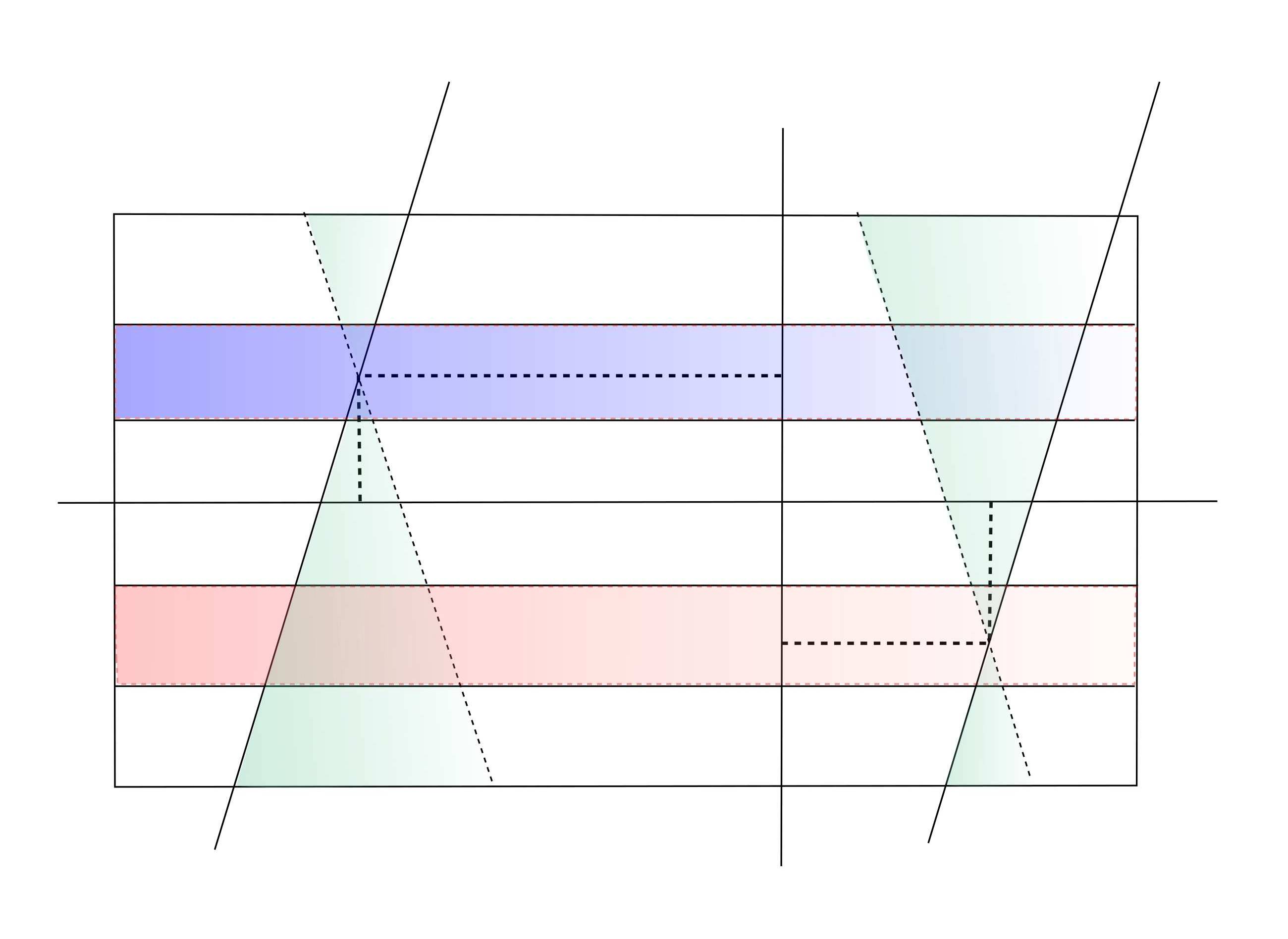}
   \put(97,35){\small{$ \mathbb{Z}$}}
     \put(61,64){\small{$ \mathbb{Y}$}}
     \put(53.5,23){{ $p_{\mu}$}}
       \put(60.5,43){{ $q_{\mu}$}}
           \put(71.5,36){{ $\tilde{p}_{\xi,\mu}$}}
             \put(23.5,30){{ $\tilde{q}_{\xi,\mu}$}}
                 \put(91,65){\small{$L^1_{\xi,\mu}$}}
                     \put(36,65){\small{$L^2_{\xi,\mu}$}}
                        \put(89,30){\small{$22$}}
                          \put(57,57){{$4$}}
                                \put(0,30){{$-40$}}
                                   \put(54,8){{$-4$}}
                                 \put(67.3,60){{\small{$C^{\mathrm{uu}}_{1/2}(P_{\xi,\mu})$}}}
                                   \put(18,6){{\small{$C^{\mathrm{uu}}_{1/2}(Q_{\xi,\mu})$}}}

 \end{overpic}
\caption{The  lines $L^1_{\xi,\mu}, L^2_{\xi,\mu}$ and the projections  in the plane $\mathbb{YZ}$ 
of the cube $\Delta$ and the 
$\mathrm{uu}$-cones at $P_{\xi,\mu}, Q_{\xi,\mu}$.}
\label{fig:6}
%\vspace{0.2cm}
\end{figure}
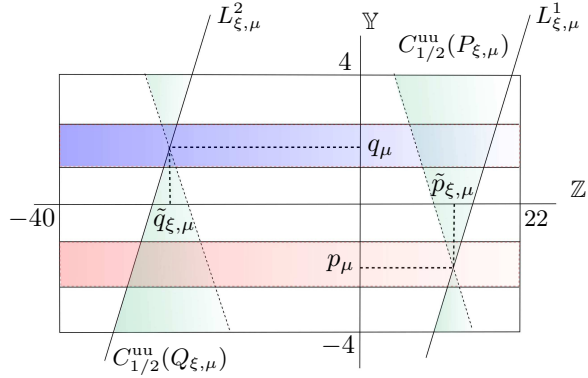
\subsection{Position of images of $\mathrm{uu}$-discs} We now study the relative positions of the images of
$\mathrm{uu}$-discs contained in $\Delta$ in Condition (\BHe). We see that this condition follows from the 
one-dimensional dynamics on the unstable center manifolds
of the saddles $P_{\xi,\mu}$ and $Q_{\xi,\mu}$, recall \eqref{e.manifolds}.

\subsubsection{One-dimensional associated dynamics}
\label{ss.one-dimensional}
Recall that
$P_{\xi,\mu}=(p_{\mu}, p_{\mu},\tilde{p}_{\xi,\mu})$ and $Q_{\mu}=(q_{\mu}, q_{\mu},\tilde{q}_{\xi,\mu})$ and 
%in  Lemma \ref{l.Hyperbolic-fixed-points} 
that the restriction of $G_{\xi,\mu}$ to the one-dimensional
center unstable
manifolds $W^{\mathrm{cu}}(P_{\xi,\mu})$, $W^{\mathrm{cu}}(Q_{\xi,\mu})$ in \eqref{e.manifolds}
 is just and affine multiplication by $\xi>1$, see
Remark~\ref{r.Invariant directions}.
Denote by 
$\phi^{r}_{\xi,\mu}$  the restriction map
$G_{\xi,\mu}|_{W^{\mathrm{cu}}(R_{\xi,\mu})\cap \Delta}$, $r=p,q$ and $R=P,Q$, that is
given by
 $$
 \phi^{r}_{\xi,\mu} \colon 
[-40,22]\to \mathbb{R},\quad
\phi^{r}_{\xi,\mu}(z)\eqdef
\xi\,z + r_\mu=
\xi\,z+(1-\xi)\, \tilde r_{\xi,\mu},\quad r=p,q,
%\phi^{-}_{\xi,\mu}
%(z)\eqdef \xi\,z+(1-\xi)\cdot\tilde q_{\xi,\mu},
$$
where we use the relation $r_\mu= (1-\xi)\, r_{\xi,\mu}$.
For $r=p,q$, consider the interval
$
\mathrm{I}^r_{\xi,\mu}\eqdef[\alpha^{r}_{\xi,\mu},
\beta^{r}_{\xi,\mu}]
$,
%and
%$
%\mathrm{I}^q_{\xi,\mu}\eqdef[\alpha^{q}_{\xi,\mu},
%\beta^{q}_{\xi,\mu}],
%$
where
\[
\begin{split}
&\alpha^{r}_{\xi,\mu}\eqdef \xi^{-1}(-40-(1-\xi)\,
\tilde r_{\mu,\xi}),
\quad \beta^{r}_{\xi,\mu}
\eqdef
\xi^{-1}(22-(1-\xi)\, \tilde r_{\xi,\mu}).
%\\
%&
%\alpha^{q}_{\xi,\mu}\eqdef \xi^{-1}(-40-(1-\xi)
%\tilde q_{\xi,\mu}),
%\quad 
%\beta^{q}_{\xi,\mu}
%\eqdef
%\xi^{-1}(22-(1-\xi)\tilde q_{\xi,\mu}).
\end{split}
\]
Note that 
 $\phi^{r}_{\xi,\mu}(\mathrm{I}^{r}_{\xi,\mu})=[-40,22]$
%
%\begin{figure}[htb]
% \centering
% \includegraphics[width=.6
% \textwidth]{SIF.pdf}
% \caption{Iterated function system} \label{fig:SIF}
%\end{figure}
and $\phi^{r}_{\xi,\mu}(\tilde r_{\xi,\mu})=\tilde r_{\xi,\mu}\in \mathrm{I}^{r}_{\xi,\mu}$. %Besides,  
%$[\tilde q_{\xi,\mu},\tilde p_{\xi,\mu}]\subset \mathrm{I}^{p}_{\xi,\mu}\cap \mathrm{I}^{q}_{\xi,\mu}$.
%Indeed, recalling that
%$\tilde r_{\xi,\mu}(1-\xi)=r_{\mu}$,
%with $p_{\mu}\in(a_{\mu},b_{\mu})$, $q_{\mu}\in(c_{\mu},d_{\mu})$ (see \eqref{e.defpuntofijo} and Remark \ref{r.fixpoint}),
%follows from  Remark~\ref{r.Rt} that
%\[\alpha^p_{\xi,\mu}=
%\xi^{-1}(-40-(1-\xi)\tilde p_{\xi,\mu})
%=
%\xi^{-1}(-40-p_{\mu})
%<
%\xi^{-1}(-40-a_{\mu})
%<\tilde q_{\xi,\mu},\]
%and 
%\[\tilde p_{\xi,\mu}<
%\xi^{-1}(22-d_{\mu})
%<
%\xi^{-1}(22-q_{\mu})=
%\xi^{-1}(22-(1-\xi)\tilde q_{\xi,\mu})
%=\beta^q_{\xi,\mu}.\]
\begin{lemma}\label{l.resuelve}
Given a $\mathrm{uu}$-disc $L$ contained  
in $\Delta$ let $L_{\mathcal{C}_{\xi,\mu}}\eqdef L\cap \mathcal{C}_{\xi,\mu}$, with $\mathcal{C}= \mathcal{A}, \mathcal{B}$.
Then $G_{\xi,\mu}(L_{\mathcal{C}_{\xi,\mu}})$ satisfies $\mathrm{(\BHe)}$.
\end{lemma}
\begin{proof}
We first show item (1) of (\BHe). Items (2), (3), and (4) are obtained similarly and their proofs will be omitted. 
%Items (5) and (6) are more difficult  and studied at the end of the proof.

From \eqref{e.WW} and \eqref{e.W}, the local stable manifolds of $P_{\xi,\mu}$ and $Q_{\xi,\mu}$ are given by 
\begin{equation}
\label{e.ultimahora}
\begin{split}
&W^{\mathrm s}_{\mathrm {loc}}(P_{\xi,\mu})=\Big\{(t+p_{\xi,\mu},p_{\xi,\mu},\tilde{p}_{\xi,\mu}): -4-p_{\xi,\mu}\le t \le 4-p_{\xi,\mu}\Big\},
\\
&W^{\mathrm s}_{\mathrm {loc}}(Q_{\xi,\mu})=\Big\{(t+q_{\xi,\mu},q_{\xi,\mu},\tilde{q}_{\xi,\mu}): -4-q_{\xi,\mu}\le t \le 4-q_{\xi,\mu}\Big\}.
\end{split}
\end{equation}

Given  a $\mathrm{uu}$-disc $L\subset \Delta$ consider the intersections
\begin{equation*}
%\label{e.L}
\begin{split}
&X^L_{\mu}\eqdef L
\cap \big(\Delta\cap \{ y=p_{\mu} \}\big)=
L_{\mathcal{A}_{\xi,\mu}}
\cap \big(\Delta\cap \{ y=p_{\mu} \}\big)
=(x_{\mu}, p_{\mu}, z_{\mu}),
\\
&\bar{X}^L_{\mu}\eqdef L
\cap \big(\Delta\cap \{ y=q_{\mu} \}\big)=
L_{\mathcal{B}_{\xi,\mu}}
\cap \big(\Delta\cap \{ y=q_{\mu} \}\big)
=(\bar{x}_{\mu}, q_{\mu}, \bar{z}_{\mu}).
\end{split}
\end{equation*}

\begin{remark}\label{r.rar}
{\em{
Recall the definitions of
the right and left  classes of $\mathrm{uu}$-discs $\mathcal{U}^r_W$ and $\mathcal{U}^\ell_W$, respectively, in Remark~\ref{r.leftandright}.
Using \eqref{e.ultimahora}
we have the following:
\begin{itemize}
\item
$L\in \mathcal{U}^\ell_{W^{\mathrm s}_{\mathrm {loc}}(P_{\xi,\mu})}$ iff $z_{\mu} < \tilde{p}_{\xi,\mu}$ and 
$L\in \mathcal{U}^r_{W^{\mathrm s}_{\mathrm {loc}}(P_{\xi,\mu})}$ iff
$z_{\mu}> \tilde{p}_{\xi,\mu}$,
\item
$L\in \mathcal{U}^\ell_{W^{\mathrm s}_{\mathrm {loc}}(Q_{\xi,\mu})}$ iff
$\bar{z}_{\mu}< \tilde{q}_{\xi,\mu}$ and 
$L\in \mathcal{U}^r_{W^{\mathrm s}_{\mathrm {loc}}(Q_{\xi,\mu})}$ iff
$\bar{z}_{\mu}>\tilde{q}_{\xi,\mu}$.
\end{itemize}
}}
\end{remark}
%\begin{remark}\label{r.rar0}
%{\em{
%Note that the planes $\{ y=r_{\mu} \}$, $r=p,q$, are invariant by $G_{\xi,\mu}$.
%}}
%\end{remark}
To prove (1) in (\BHe), take any $L\in \mathcal{U}^r_{W^{\mathrm s}_{\mathrm {loc}}(P_{\xi,\mu})}$.  We will see that
$G_{\xi,\mu}(L_{\mathcal{A}_{\xi,\mu}})\in \mathcal{U}^r_{W^{\mathrm s}_{\mathrm {loc}}(P_{\xi,\mu})}$. 
By Remark~\ref{r.rar},
the point $X_\mu^L=(x_{\mu}, p_{\mu}, z_{\mu})$
satisfies $z_{\mu}>\tilde p_{\xi,\mu}$.
Note that
$$
G_{\xi,\mu}(X^L_{\mu})=\big(p_\mu, p_\mu, \phi^{p}_{\xi,\mu}(z_{\mu})\big)=
\big(p_\mu, p_\mu, \xi z_{\mu}+ (1-\xi) \tilde p_{\xi,\mu}\big).
$$
Since $z_{\mu}> \tilde p_{\xi,\mu}$ it follows that
$\phi^{p}_{\xi,\mu}(z_{\mu})>\tilde p_{\xi,\mu}$. Remark \ref{r.rar} now implies that $G_{\xi,\mu}(L_{\mathcal{A}_{\xi,\mu}})\in \mathcal{U}^r_{W^{\mathrm s}_{\mathrm {loc}}(P_{\xi,\mu})}$.

Since items (5) and (6) of  (\BHe) are analogous we just prove 
item (5).
%, omitting the proof of (6). 
We  just need to check that if
 $L\in \mathcal{U}^r_{W^{\mathrm s}_{\mathrm {loc}}(P_{\xi,\mu})}$
or $L \cap W^\mathrm{s}_{\mathrm{loc}} (P_{\xi,\mu})\neq\emptyset$ then
$G_{\xi,\mu}(L_{\mathcal{B}_{\xi,\mu}})\in \mathcal{U}^r_{W^{\mathrm s}_{\mathrm {loc}}(P_{\xi,\mu})}$. 
%{\textcolor{red}{As in the previous section, we reduced the analysis to the two dimensional case.}}
%\margem{Q: creo que no redujimos nada...}

%We observe the following. 

\begin{remark}
{\em{
Consider the projection  $\Pi_1(x,y,z)=(y,z)$ and note that
 $$
 \Pi_1\big(L\cap \{y\ge p_{\mu}\}\big)\subset \Gamma_{\xi,\mu}\eqdef \big\{(y,z): z\ge z^*_{\xi,\mu}(y)\big\},
 $$
see Figure \ref{fig:7}. Moreover, 
$
\Pi_1(L_{\mathcal{B}_{\xi,\mu}})\subset \Gamma_{\xi,\mu}\cap \Pi_1(\mathcal{B}_{\xi,\mu}).
$
}}
\end{remark}

Note that the worst case to prove (5) in  (\BHe) occurs when $L$ is contained in the plane $\mathbb{YZ}$ 
and equal to the straight line $L^*_{\xi,\mu}$ in the plane $\mathbb{YZ}$ through $(p_{\mu},\tilde{p}_{\xi,\mu})$ given by
\begin{equation} 
\label{e.definitiontildezeta}
L^*_{\xi,\mu} \eqdef\Big\{ \big(y, z^*_{\xi,\mu}(y)\big) \colon  z^*_{\xi,\mu}(y)=-\frac1{2}(y-p_{\mu})+\tilde{p}_{\xi,\mu}, \,\, y\in \mathbb{R}\Big\}.
\end{equation}
Consider 
the segment of  $L^*_{\xi,\mu}$ given by (see Figure \ref{fig:7})
$$\gamma_{\xi,\mu}\eqdef\big\{\big(y,z^*_{\xi,\mu}(y)\big):y\in\J_{\mu}\big\}
\subset L^*_{\xi,\mu} \cap\Pi_1(\mathcal{B}_{\xi,\mu}) 
$$
and the point 
$\tilde{z}_{\xi,\mu}$ defined by
\begin{equation}
\label{e.geo}
g_{\xi,\mu}(\gamma_{\xi,\mu})\cap \{y=p_{\mu}\}=\{(p_{\mu},\tilde{z}_{\xi,\mu})\},
\end{equation}
where  the endomorphism $g_{\xi,\mu}$ obtained by  projecting  $G_{\xi,\mu}$ 
into the plane $\mathbb{YZ}$
defined in \eqref{e.g}. 
By   Remark \ref{r.rar}
 to get $G_{\xi,\mu}(L_{\mathcal{B}_{\xi,\mu}})\in \mathcal{U}^r_{W^{\mathrm s}_{\mathrm {loc}}(P_{\xi,\mu})}$ it is sufficient to show that $\tilde{z}_{\xi,\mu}>\tilde{p}_{\xi,\mu}$.
\begin{figure}
\centering
\begin{overpic}[scale=.3,bb=0 0 1333 681,tics=5
  ]{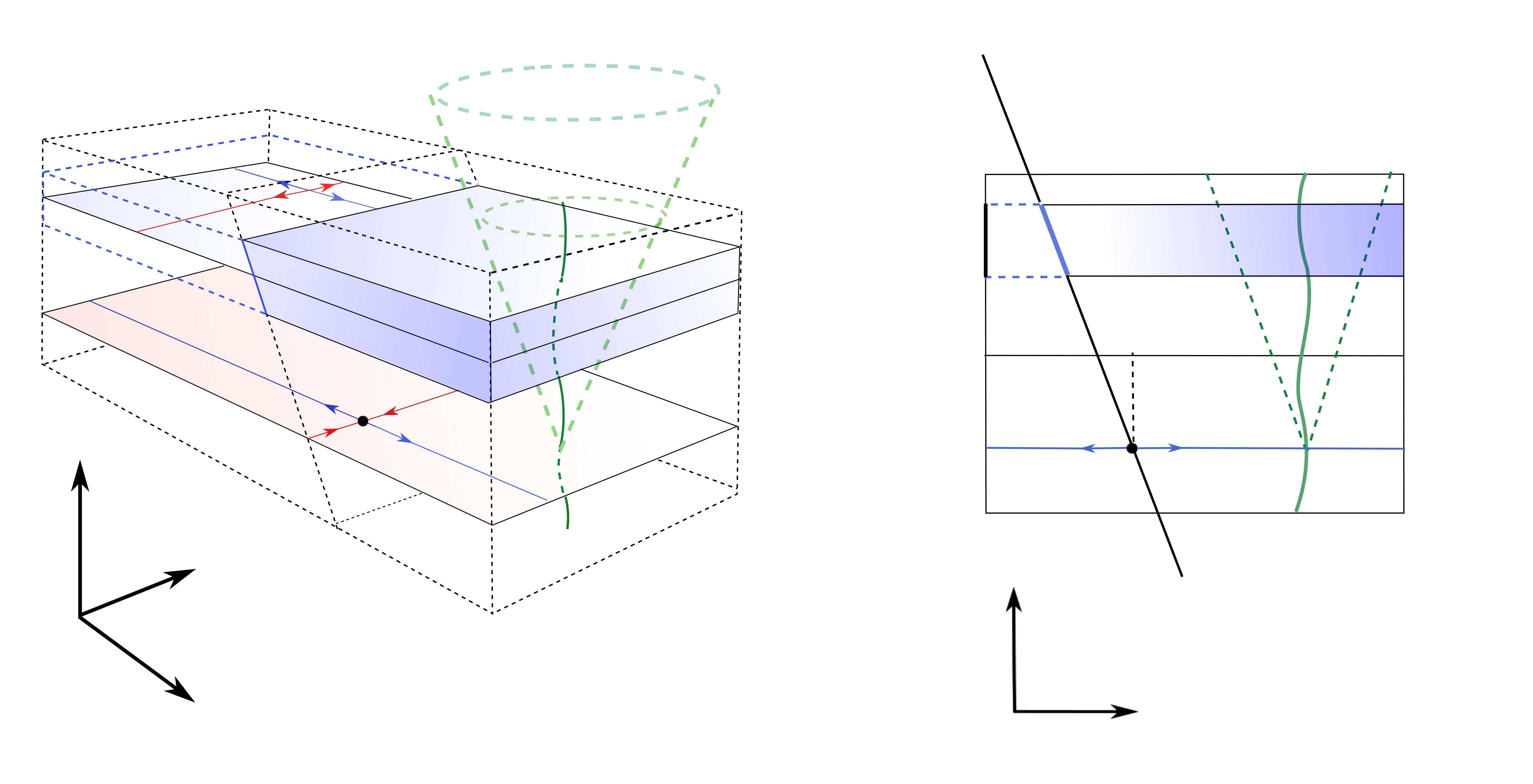}
  \put(13,13){\small{$ \mathbb{X}$}}
     \put(6,19){\small{$ \mathbb{Y}$}}
      \put(13,6){\small{$ \mathbb{Z}$}}
      \put(21.5,25.5){{$P_{\xi,\mu}$}}
       \put(33.5,35){{ $L$}}
        \put(49,32){{$\mathcal{B}_{\xi,\mu}$}}
        \put(33,49){{$\mathcal{C}^{\mathrm{uu}}_{1/2}(X_{\mu})$}}
        \put(36,21){{\small{$\bullet$}}}
          \put(38,21.5){{$X_\mu$}}
        \put(61,35){{$\J_{\mu}$}}
        \put(67,10){\small{$ \mathbb{Y}$}}
      \put(73,6){\small{$ \mathbb{Z}$}}
        \put(70,35){{$\gamma_{\xi,\mu}$}}
     \put(60,21.6){{ $p_{\mu}$}}
      % \put(72,33){{ $q_{\mu}$}}
           \put(71,29.5){{ $\tilde{p}_{\xi,\mu}$}}
            % \put(60,25){{ $\tilde{q}_{\xi,\mu}$}}
                 \put(66,45){\small{$L^*_{\xi,\mu}$}}
                        \put(81,41){{ $\Pi_1(L)$}}
            \put(92.6,34.3){{$\Pi_1(\mathcal{B}_{\xi,\mu})$}}           
                        \put(92.3,27){\small{$22$}}
                          \put(61.8,39.5){{$4$}}
                           \put(60,16){{$-4$}}
                                \put(62,27){{$0$}}
                                 \put(30,2){{$(a)$}}
                                \put(80,2){{$(b)$}}    
 \end{overpic}
\caption{$(a)$ 
 $L$ is a uu-disc in $\mathcal{U}^r_{W^{\mathrm s}_{\mathrm {loc}}(P_{\xi,\mu})}$. $(b)$ 
 Projection of $L$ in the plane $\mathbb{YZ}$.}
\label{fig:7}
%\vspace{0.2cm}
\end{figure}
\begin{cl} \label{cl.keep}
It holds $\tilde{z}_{\xi,\mu}>\tilde{p}_{\xi,\mu}$  for every $(\xi,\mu)\in\mathcal{P}$. 
\end{cl}
\begin{proof}
The intersection \eqref{e.geo}
is defined by the conditions
$$
(p_{\mu}, \tilde{z}_{\xi,\mu})=(y^2+\mu,\xi\, z^*_{\xi,\mu}(y)+y), \quad y>0.
$$
Recalling the definition of $z^*_{\xi,\mu}(y)$ in \eqref{e.definitiontildezeta} we get
$$
\tilde{z}_{\xi,\mu}=\xi \, \tilde{z}^*_{\xi,\mu}\big(\sqrt{p_\mu-\mu}\big) +  \sqrt{p_\mu-\mu} =\frac{\xi}{2}\, p_{\mu}+\Big(1-\frac{\xi}{2}\Big)\sqrt{p_\mu-\mu}+\xi\,\tilde p_{\xi,\mu}.
$$
Hence
$$
\tilde{z}_{\xi,\mu}-\tilde p_{\xi, \mu}
=
\frac{\xi}{2}\, p_{\mu}+\Big(1-\frac{\xi}{2}\Big)\sqrt{p_\mu-\mu}+(\xi-1)\,\tilde p_{\xi,\mu}.
$$
The estimates in \eqref{e.est} and the choice of $(\xi,\mu)\in \mathcal{P}$ imply that
$$
\frac{\xi}{2}\,p_{\mu}>-1.6065, 
\quad
\Big(1-\frac{\xi}{2}\Big)\sqrt{p_\mu-\mu}>1.014
,\quad
(\xi-1)\,\tilde p_{\xi,\mu}>2.34.
$$
These inequalities imply that
$\tilde{z}_{\xi,\mu}-\tilde p_{\xi,\mu}>0$, proving the claim.
%
% To obtain $z^2_{\xi,\mu}>\tilde p_{\xi,\mu}$, we estimate 
%\begin{equation*}
%\begin{split}
%\xi z_2+y-\tilde p_{\xi,\mu}&=\xi\Big[
%\frac1{2}(\sqrt{p_\mu-\mu}-p_{\mu})+z_\mu
%\Big]+\sqrt{p_\mu-\mu}-\tilde p_{\xi,\mu}\\
%&>\xi\Big[
%\frac1{2}(\sqrt{p_\mu-\mu}-p_{\mu})+\tilde p_{\xi,\mu}
%\Big]+\sqrt{p_\mu-\mu}-\tilde p_{\xi,\mu}\\
%&=
%-\frac{\xi}{2}p_{\mu}+\Big(1+\frac{\xi}{2}\Big)\sqrt{p_\mu-\mu}+(\xi-1)\tilde p_{\xi,\mu}.
%\end{split}
%\end{equation*}
%From the previous estimates, it follows that each term in this last inequality is strictly positive and thus $\xi z_2+y>\tilde p_{\xi,\mu}$, ending the prove of the sub-claim.
\end{proof}
%By the claim,  $G_{\xi,\mu}(L_{\mathcal{B}_{\xi,\mu}})\in \mathcal{U}^r_{W^{\mathrm s}_{\mathrm {loc}}(P_{\xi,\mu})}$,
The proof of the lemma is now complete. 
%ending  the proof of the lemma.
\end{proof}

\subsection{Position of images of $\mathrm{uu}$-discs in between}
\label{ss.between} 
Condition (\BHf) is given by Lemma~\ref{l.below} below. First, recall the definition of the family of disks in between
$\mathcal{U}^b \eqdef \mathcal{U}^\ell_{W^{\mathrm s}_{\mathrm {loc}}(P)}
\cap
\mathcal{U}^r_{W^{\mathrm s}_{\mathrm {loc}}(Q)}.$

\begin{lemma}
\label{l.below}
Consider any $L\in\mathcal{U}^b$. Then  either $G_{\mu,\xi}(L_{\mathcal{A}_{\xi,\mu}})$ or $G_{\mu,\xi}(L_{\mathcal{B}_{\xi,\mu}})$ contains a $\mathrm{uu}$-disc in $\mathcal{U}^b$.
\end{lemma}

\begin{proof}
Consider $L\in\mathcal{U}^b$.
By item (2) in (\BHe),
if $G_{\xi,\mu}(L_{\mathcal{A}_{\xi,\mu}})\in \mathcal{U}^r_{W^{\mathrm s}_{\mathrm {loc}}(Q_{\xi,\mu})}$
then $G_{\xi,\mu}(L_{\mathcal{A}_{\xi,\mu}})\in \mathcal{U}^b$ and we are done.
Similarly,
by item (3) in (\BHe),
if $G_{\xi,\mu}(L_{\mathcal{B}_{\xi,\mu}})\in \mathcal{U}^\ell_{W^{\mathrm s}_{\mathrm {loc}}(P_{\xi,\mu})}$
then $G_{\xi,\mu}(L_{\mathcal{B}_{\xi,\mu}})\in \mathcal{U}^b$ and we are done.
Thus in what follows
we  argue by contradiction assuming that: 
\begin{itemize}
\item[a)]
$G_{\xi,\mu}(L_{\mathcal{A}_{\xi,\mu}})\in \mathcal{U}^\ell_{W^{\mathrm s}_{\mathrm {loc}}(Q_{\xi,\mu})}$ or
intersects $W^{\mathrm s}_{\mathrm {loc}}(Q_{\xi,\mu})$ and
\item[b)] 
$G_{\xi,\mu}(L_{\mathcal{B}_{\xi,\mu}})\in \mathcal{U}^r_{W^{\mathrm s}_{\mathrm {loc}}(P_{\xi,\mu})}$
or intersects 
$W^{\mathrm s}_{\mathrm {loc}}(P_{\xi,\mu})$.
\end{itemize}

To prove the lemma 
we need some auxiliary constructions.
Consider 
the point
 $Y^L_\mu=(x_{\mu}, a_{\mu}, z_{\mu})\eqdef  L \cap \{y=a_{\mu}\}$, where $a_\mu$ is defined in \eqref{e.ys}. In the plane $\mathbb{YZ}$,
take the  auxiliary straight line $\widehat{L}_{\mu}$ through $(a_\mu,z_{\mu})$
 given by (see Figure \ref{fig:8})
\begin{equation*}
%R^1_{\mu} : z^1_{\mu}(y)=-\frac1{2}(y-p_{\mu})+z_\mu,\quad
\widehat{L}_{\mu}   \eqdef \Big\{\big(y, z_\mu^a(y)\big) \colon z^a_\mu(y)=\frac1{2}(y-a_{\mu})+z_\mu, \,\, y\in \mathbb{R}\Big\}.
\end{equation*}
Observe that $ \widehat{L}_{\mu}  \subset \partial\Pi_1\big(\mathcal{C}^{\mathrm{uu}}_{1/2}(Y^L_{\mu})\big)$. 
 Consider the sub segments  of $\widehat{L}_{\mu}$ given by (see Figure \ref{fig:8})
\begin{equation*}
\widehat L^I_{\mu}\eqdef
\big\{\big(y,z^a_{\mu}(y)\big): y\in\I_{\mu}\big\}\quad 
\mbox{and}\quad
\widehat L^J_{\mu}\eqdef
\big\{\big(y,z^a_{\mu}(y)\big): y\in\J_{\mu}\big\}.
\end{equation*}

\begin{figure}
\centering
\begin{overpic}[scale=.35,bb=0 0 738 543,tics=5
  ]{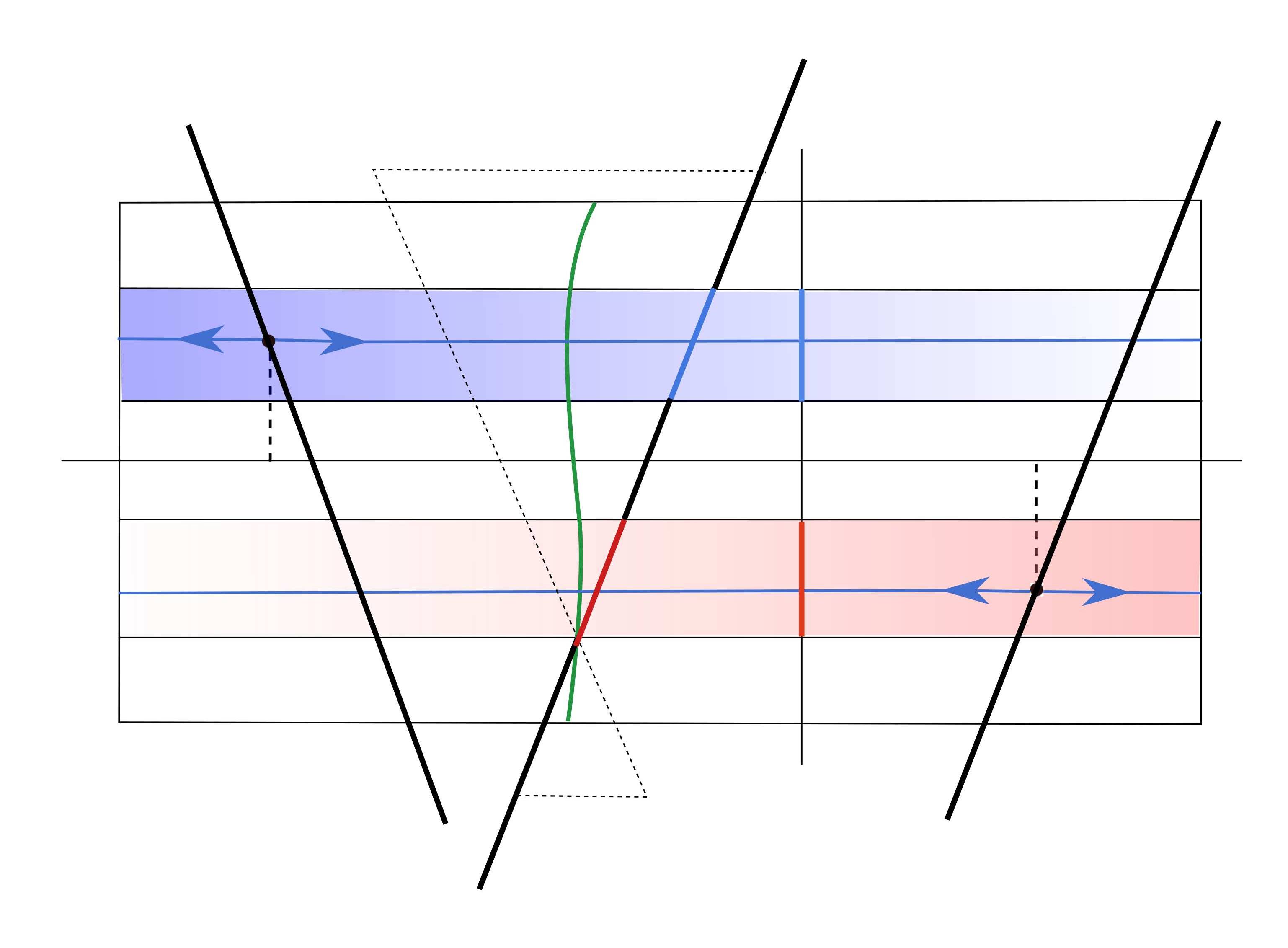}
%  \put(13,13){\small{$ \mathbb{X}$}}
%     \put(6,19){\small{$ \mathbb{Y}$}}
%      \put(13,6){\small{$ \mathbb{Z}$}}
%      \put(21.5,25.5){{$P_{\xi,\mu}$}}
      \put(39,53){{ $L$}}
%        \put(49,32){{$\mathcal{B}_{\xi,\mu}$}}
%        \put(61,35){{$\J_{\mu}$}}
%        \put(67,10){\small{$ \mathbb{Y}$}}
%      \put(73,6){\small{$ \mathbb{Z}$}}
       \put(49,27){{$\widehat{L}^I_{\mu}$}}
        \put(47,46){{$\widehat{L}^J_{\mu}$}}
    \put(63,26.6){{ $\I_{\mu}$}} 
         \put(62,46){{ $\J_{\mu}$}}
           \put(76,40){{ $\tilde{p}_{\xi,\mu}$}}
            \put(16,34){{ $\tilde{q}_{\xi,\mu}$}}
                \put(55,65){\small{$\widehat L_{\mu}$}}
%                        \put(81,41){{ $\Pi_1(L)$}}
%            \put(92.6,34.3){{$\Pi_1(\mathcal{B}_{\xi,\mu})$}}           
                        \put(95,33){\small{$22$}}
                          \put(65,60){{$4$}}
                           \put(63,13){{$-4$}}
                                \put(0,33){{$-40$}}
                                                         \put(85,62){{$\widehat{L}^p_{\xi,\mu}$}}
                                                         \put(17,62){{$\widehat{L}^q_{\xi,\mu}$}}

%                                \put(80,2){{$(b)$}}    
 \end{overpic}
\caption{
 The segments $\widehat{L}^a_\mu$  and $\widehat{L}^b_\mu$ and the lines $\widehat{L}^p_{\xi,\mu}$ and  $\widehat{L}^q_{\xi,\mu}$.}
\label{fig:8}
\end{figure}
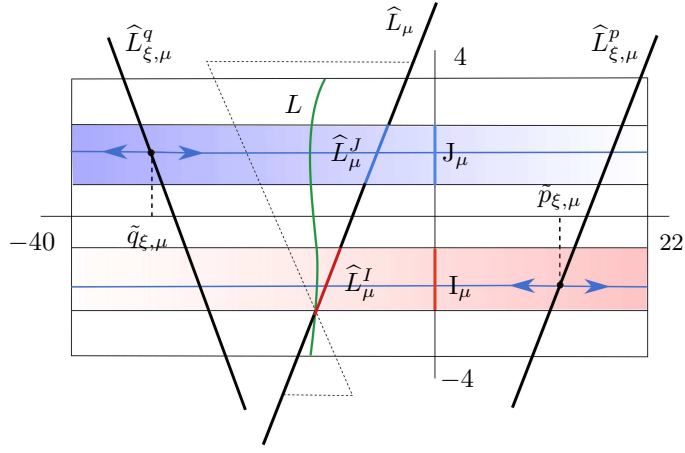

Recall that $P_{\xi,\mu}=(p_{\mu}, p_{\mu},\tilde{p}_{\xi,\mu})$ and $Q_{\mu}=(q_{\mu}, q_{\mu},\tilde{q}_{\xi,\mu})$ and
consider the straight lines $\widehat{L}^p_{\xi,\mu}$  and  $\widehat{L}^q_{\xi,\mu}$ contained in
 $\partial\Pi_1\big(\mathcal{C}^{\mathrm{uu}}_{1/2}(P_{\xi,\mu})\big)$ and 
$\partial\Pi_1\big(\mathcal{C}^{\mathrm{uu}}_{1/2}(Q_{\xi,\mu})\big)$, respectively,
given by
\begin{equation*}
\begin{split}
\widehat{L}^p_{\xi,\mu}   \eqdef \Big\{ \big(y, z^p(y)\big) \colon z^p(y)&=\frac1{2}(y-p_{\mu})+\tilde{p}_{\xi,\mu}, \,\, y\in \mathbb{R}\Big\},
\\
\widehat{L}^q_{\xi,\mu}   \eqdef \Big\{ \big(y, z^q (y)\big) \colon z^q(y)&=-\frac1{2}(y-q_{\mu})+\tilde{q}_{\xi,\mu}, \,\, y\in \mathbb{R}\Big\}.
\end{split}
\end{equation*}
Finally, 
consider the following subsets of $\Delta$
$$
\Sigma^p_{\xi,\mu}\eqdef \Big([-4,4]\times \widehat L^{p}_{\xi,\mu}\Big) \cap \Delta, \quad
 \Sigma^{q}_{\xi,\mu}\eqdef \Big([-4,4]\times \widehat L^{q}_{\xi,\mu}\Big) \cap \Delta.
$$
Observe that $\Delta \setminus \Sigma^r_{\xi,\mu}$, $r=p,q$, consists of two connected components.
We let $\Delta^Q_{\xi,\mu,\mathrm{right}}$ the connected component of  $\Delta \setminus \Sigma^q_{\xi,\mu}$ containing $P_{\xi,\mu}$ and  by  $\Delta^Q_{\xi, \mu,\mathrm{left}}$ the other component. Similarly, we let 
$\Delta^P_{\xi, \mu,\mathrm{left}}$ the connected component of  $\Delta \setminus \Sigma^p_{\xi,\mu}$ containing $Q_{\xi,\mu}$ and  by  $\Delta^P_{\xi,\mu,\mathrm{right}}$ the other component.

After these preliminary constructions, we are now ready to prove the lemma.
Note that by  Remark \ref{r.reversing} 
``$G_{\xi,\mu}\big([-4,4]\times \widehat{L}^I_{\mu}\big)$ is at the left of
$G_{\xi,\mu}(L_{\mathcal{A}_{\xi,\mu}})$'' and
``$G_{\xi,\mu}\big([-4,4]\times \widehat{L}^J_{\mu}\big)$ is at the right of
$G_{\xi,\mu}(L_{\mathcal{B}_{\xi,\mu}})$''. Therefore
\begin{itemize}
\item condition (a) implies that
$G_{\xi,\mu}\big([-4,4]\times \widehat{L}^I_{\mu}\big)\subset \textrm{closure}(\Delta^Q_{\xi,\mu, \mathrm{left}})$,
\item condition (b) implies that
$G_{\xi,\mu}\big([-4,4]\times \widehat{L}^J_{\mu}\big)\subset \textrm{closure}(\Delta^P_{\xi,\mu, \mathrm{right}})$.
\end{itemize}
We now see that these two conditions cannot hold simultaneously.
Consider  $\omega^{I}_{\xi,\mu}, \omega^{J}_{\xi,\mu}\in \mathbb{Z}$ 
given by
$$
g_{\xi,\mu}\big(\widehat{L}^{I}_{\mu}\big)\cap \{y=q_\mu\}=(q_\mu,\omega^{I}_{\xi,\mu})
\quad\mbox{and}\quad 
g_{\xi,\mu}\big(\widehat{L}^{J}_{\mu}\big)\cap \{y=p_\mu\} =(p_\mu,\omega^{J}_{\xi,\mu}).
$$
Arguing as in Claim \ref{cl.keep}, we get
$$
\omega^{I}_{\xi,\mu}=\frac{\xi}{2}(q_\mu-a_{\mu})+\xi z_\mu+ q_\mu\quad\mbox{and}\quad \omega^{J}_{\xi,\mu}=\xi z_\mu+ a_\mu.
$$ 
On the other hand, our assumptions and Remark~\ref{r.rar} imply that $\omega^{I}_{\xi,\mu}\le \tilde q_{\xi,\mu}$ and $\omega^{J}_{\xi,\mu}\ge \tilde p_{\xi,\mu}$.
Thus
$$
|\tilde q_{\xi,\mu}-\tilde p_{\xi,\mu}|\le 
|\omega^{I}_{\xi,\mu}-\omega^{J}_{\xi,\mu}|
\le 
\Big(\frac{\xi}{2}+1\Big) | q_{\mu}- a_{\mu}| \le 12.16,
$$
where the last inequality follows from the estimates in \eqref{e.est} and \eqref{e.estim}.
Since, also by \eqref{e.est}, we have that
$|\tilde q_{\xi,\mu}-\tilde p_{\xi,\mu}|\in [31.4, 35,6]$ we derive a contradiction, completing the proof of the lemma.
\end{proof}
\bibliographystyle{siam}

\end{document}